\documentclass[11pt]{article}
\usepackage[letterpaper,twoside,outer=1in,vmargin=1in,]{geometry}
\usepackage{setspace}

\usepackage{amsmath,amsfonts,amsthm,url,graphicx,subcaption,float,booktabs,bbm,enumitem}
\usepackage{algorithm,algpseudocode}
\usepackage{tikz,caption,forest}\usepackage{multicol}
\tikzset{
	treenode/.style = {shape=rectangle, rounded corners,
		draw, align=center, 
		minimum height=2ex, text depth=0.25ex,
		top color=white, bottom color=blue!20},
	root/.style     = {treenode, font=\Large\rmfamily, bottom color=red!30},
	env/.style      = {treenode, font=\ttfamily\normalsize},
}
\usetikzlibrary{arrows.meta}
\algnewcommand{\Initialize}[1]{%
	\State \textbf{Initialize:}
	\Statex \hspace*{\algorithmicindent}\parbox[t]{.8\linewidth}{\raggedright #1}
}
\algnewcommand{\algorithmicgoto}{\textbf{go to}}%
\algnewcommand{\Goto}[1]{\algorithmicgoto~step~\ref{#1}}%
\usepackage[normalem]{ulem}

\def \bR {\mathbb{R}}
\def \bZ {\mathbb{Z}}

\def \conv {\text{conv}}

\def \proj {\text{proj}}
\def \dim {\text{dim}}

\def \MIP {\emph{MIP}}
\def \OBJ {\emph{OBJ}}
\def \DWB {\emph{DWB}}
\def \STR {\emph{STR}}
\def \HYB {\emph{HYB}}
\def \DkT#1 {\emph{D}$#1$\emph{T}}
\def \T {\bar{t}}

\newtheorem{thm}{Theorem}

\newtheorem{prop}[thm]{Proposition}
\newtheorem{cor}[thm]{Corollary}

\theoremstyle{plain}
\newtheorem{defn}{Definition}

\newtheorem{exmp}{Example}

\newcommand{\myalpha}{w}

\author{Rui Chen$^1$, Oktay G\"{u}nl\"{u}k$^2$, Andrea Lodi$^1$\\
$^1$ Cornell Tech, Cornell University (\{rui.chen,andrea.lodi\}@cornell.edu)\\
$^2$ School of ORIE, Cornell University (oktay.gunluk@cornell.edu)}
\title{Recovering Dantzig-Wolfe Bounds by Cutting Planes}

\begin{document}
\maketitle

\begin{abstract}
 Dantzig-Wolfe (DW) decomposition is a well-known technique in mixed-integer programming (MIP) for decomposing and convexifying constraints to obtain potentially strong dual bounds. We investigate cutting planes that can be derived using the DW decomposition algorithm and show that these cuts can provide the same dual bounds as DW decomposition. 
More precisely, we generate one cut for each DW block, and when combined with the constraints in the original formulation, these cuts imply the objective function cut one can simply write using the DW bound.
This approach typically leads to a formulation with lower dual degeneracy that consequently has a better computational performance when solved by standard MIP solvers in the original space. We also discuss how to strengthen these cuts to improve the computational performance further. We test our approach on the Multiple Knapsack Assignment Problem and the Temporal Knapsack Problem, and show that the proposed cuts are helpful in accelerating the solution time  without the need to implement branch and price.

\end{abstract}

\section{Introduction}

In this paper, we present  a computationally effective approach for generating cutting planes from Dantzig-Wolfe (DW) decomposition \cite{dantzig1960decomposition} to enhance branch and cut in the space of original variables. We focus on mixed-integer (linear) programs (MIPs) with the following structure: 
\begin{equation}\label{original_MIP}
	\begin{aligned}
		z^*:=\min\ &c^\top x\\
		\text{s.t. }&x_{I(j)}\in P^j, \quad j\in J :=\{1,\ldots,q\},\\
		&Ax\geq b,~x\in X,
	\end{aligned}
\end{equation}
where the index set $I(j)\subseteq\{1,\ldots,n\}$ contains the indices of the variables in ``block" $j\in J$, and we use the notation $x_{I}$ to denote the subvector of $x$ with indices in $I$. 
The set  $P^j=\{y\in\bR^{|I(j)|}:G^jy\geq g^j \}$ is a polyhedral set for $j\in J$; and  $X\subseteq\bR^n$ represents integrality constraints on some of the variables. We do not assume the index sets  $I(j)$ for $j\in J$ to be disjoint, see Figure \ref{fig:the blocks}. 

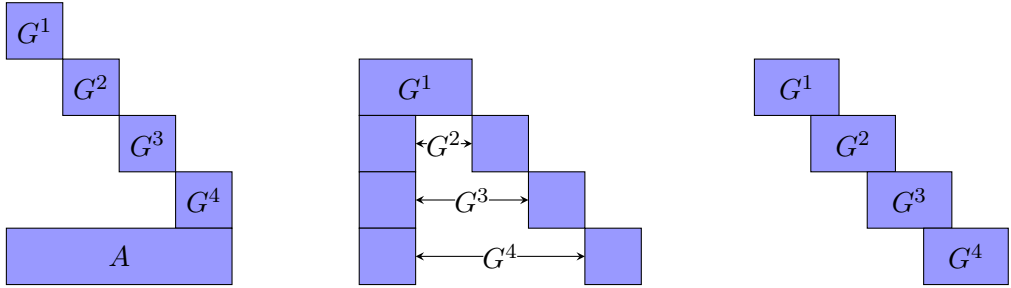
\begin{figure}[tb]
	\centering
	\caption{Constraint Matrices of MIPs With Different Types of Block Structures (Left: Loosely Coupled, Middle: Two-Stage, Right: Overlapping)}
	\begin{subfigure}{0.3\textwidth}
		\centering
		\begin{tikzpicture}[scale=0.75]
			\filldraw[fill=blue!40!white, draw=black] (0,0) rectangle (1,-1);
			\node at (0.5,-0.5) {$G^1$};
			\filldraw[fill=blue!40!white, draw=black] (1,-1) rectangle (2,-2);
			\node at (1.5,-1.5) {$G^2$};
			\filldraw[fill=blue!40!white, draw=black] (2,-2) rectangle (3,-3);
			\node at (2.5,-2.5) {$G^3$};
			\filldraw[fill=blue!40!white, draw=black] (3,-3) rectangle (4,-4);
			\node at (3.5,-3.5) {$G^4$};
			\filldraw[fill=blue!40!white, draw=black] (0,-4) rectangle (4,-5);
			\node at (2,-4.5) {$A$};
		\end{tikzpicture}
	\end{subfigure}
	\begin{subfigure}{0.3\textwidth}
		\centering
		\begin{tikzpicture}[scale=0.75]
			\filldraw[fill=blue!40!white, draw=black] (-1,0) rectangle (1,-1);
			\filldraw[fill=blue!40!white, draw=black] (1,-1) rectangle (2,-2);
			\filldraw[fill=blue!40!white, draw=black] (2,-2) rectangle (3,-3);
			\filldraw[fill=blue!40!white, draw=black] (3,-3) rectangle (4,-4);
			\filldraw[fill=blue!40!white, draw=black] (-1,-1) rectangle (0,-2);
			\filldraw[fill=blue!40!white, draw=black] (-1,-2) rectangle (0,-3);
			\filldraw[fill=blue!40!white, draw=black] (-1,-3) rectangle (0,-4);
			\node at (0,-0.5) {$G^1$};
			\node at (0.5,-1.5) {$G^2$};
			\node at (1,-2.5) {$G^3$};
			\node at (1.5,-3.5) {$G^4$};
			\draw [-stealth](0.2,-1.5) -- (0,-1.5);
			\draw [-stealth](0.8,-1.5) -- (1,-1.5);
			\draw [-stealth](0.7,-2.5) -- (0,-2.5);
			\draw [-stealth](1.3,-2.5) -- (2,-2.5);
			\draw [-stealth](1.2,-3.5) -- (0,-3.5);
			\draw [-stealth](1.8,-3.5) -- (3,-3.5);
		\end{tikzpicture}
	\end{subfigure}
	\begin{subfigure}{0.3\textwidth}
		\centering
		\begin{tikzpicture}[scale=0.75]
			\filldraw[fill=blue!40!white, draw=black] (0,0) rectangle (1.5,-1);
			\filldraw[fill=blue!40!white, draw=black] (1,-1) rectangle (2.5,-2);
			\filldraw[fill=blue!40!white, draw=black] (2,-2) rectangle (3.5,-3);
			\filldraw[fill=blue!40!white, draw=black] (3,-3) rectangle (4.5,-4);
			\node at (0.75,-0.5) {$G^1$};
			\node at (1.75,-1.5) {$G^2$};
			\node at (2.75,-2.5) {$G^3$};
			\node at (3.75,-3.5) {$G^4$};
		\end{tikzpicture}
	\end{subfigure}
	\label{fig:the blocks}
\end{figure}

DW decomposition  was originally developed for solving large-scale linear programs with loosely coupled blocks and later extended to MIPs with similar block structures to obtain strong dual bounds. 
Typically, these so-called DW bounds are stronger than the linear programming (LP) relaxation bounds as they  exploit the block structure to  partially convexify the solution space using integrality information.
DW decomposition has been found to be effective in various applications, such as transportation \cite{lubbecke2003engine}, traffic scheduling \cite{rios2010massively}, energy \cite{anjos2018decentralized}, and multi-stage stochastic planning \cite{singh2009dantzig}.


Computing the DW bound requires reformulating the MIP using the extreme points and extreme rays of the mixed-integer hulls of the block problems. 
Consequently, while the DW reformulation approach often leads to good dual bounds, using this to solve the MIP exactly requires specialized techniques that are not readily present in off-the-shelf solvers.
In other words, to  exploit DW decomposition one needs to solve the continuous relaxation of the reformulated MIP by column generation followed by an \emph{ad-hoc} branching step \cite{vanderbeck2000dantzig}, after which one has to again resort to column generation, leading to an algorithmic framework called branch and price \cite{barnhart1998branch}. 
While this approach has been successfully implemented in some special cases, most notably  for vehicle routing problems \cite{vrpsolver-pessoa}, generic branch-and-price solvers, including ABACUS \cite{junger2000abacus}, G12 \cite{puchinger2008high}, GCG \cite{gamrath2010experiments}, DIP \cite{galati2012computational}, BaPCod \cite{sadykov2021bapcod}, to name a few, are still in their infancy with respect to solving general MIPs with block structure.
However, the most up-to-date (and always improving) MIP solvers are based on the  branch-and-cut (or cut-and-branch) scheme \cite{lodi2010mixed}.

%

In this paper, we aim to make a step towards bridging this gap by developing a new scheme to incorporate the  dual bounds produced by DW reformulations into the standard branch-and-cut framework. More precisely, we first compute the  DW reformulation bound $z_D$ of the  MIP \eqref{original_MIP} at the root node and then generate cutting planes to incorporate this bound into the original formulation to solve the problem to optimality in the space of the original variables using standard MIP technology. Our approach, therefore, requires column generation only at the root node and not throughout the enumeration tree.
Apart from our proposed approach, one trivial method to enforce the DW bound $z_D$ in the original MIP is to augment the formulation by adding the \emph{objective function cut}, $c^\top x \geq z_D,	$
which is known (folklore) to be not only computationally ineffective but also numerically unstable. To the best of our knowledge, however, no detailed theoretical or computational investigation has been conducted on the objective function cut. 
Earlier cutting plane approaches include \cite{ralphs2003capacitated,ralphs2005decomposition,avella2010computational}, where valid inequalities are added iteratively to separate candidate solutions from the DW relaxation polytope. Since we only add one round of valid inequalities, our approach tends to be less computational intensive and leads to a less complex MIP formulation in the end.


\paragraph{Paper Contributions.} Our first contribution is to confirm the instability of using the objective cut by mostly attributing it to dual degeneracy. We then show how to overcome this issue by using a family of cutting planes that essentially decompose the objective function cut. This leads to a formulation with lower dual degeneracy, and consequently a better computational performance when solved by standard MIP solvers in the original space. Moreover, we propose two distinct ways to strengthen these cuts to improve their computational performance further.
As a case study, we test our approach on the Multiple Knapsack Assignment Problem and the Temporal Knapsack Problem to show that the proposed cuts are helpful in accelerating the solution time without the need to implement branch and price. 
Finally, we consider a standard multi-thread computational environment and present a simple machine learning approach to identify problem instances when our cutting plane approach has the potential to improve computation time.
This framework is simple and general, holding the promise of being implementable in a relatively easy way in general-purpose MIP solvers.

\paragraph{Paper Organization.} The remainder of the paper is organized as follows. In Section \ref{sec:preliminary}, we provide some preliminaries on DW decomposition and introduce what we call DW Block cuts. In Section \ref{sec:DWintoBnC}, we propose our new approach for recovering the DW bound in the original space using these cuts and discuss how they relate to the objective function cut. In Section \ref{sec:Lag}, we discuss Lagrangian relaxation as an alternative approach for computing the DW bound and generating cutting planes. Some techniques for strengthening the proposed cuts are presented in Section \ref{sec:Strengthening}. Section \ref{sec:MKAP} reports our computational investigation while Section \ref{sec:hybrid} proposes and tests the multi-thread hybrid algorithm. Finally, some short conclusions are drawn in Section \ref{sec:end}.

\section{Preliminaries}
\label{sec:preliminary}
We assume MIP \eqref{original_MIP} is feasible.
Throughout the paper, we call formulation \eqref{original_MIP} the original formulation of the MIP and call constraints $Ax\geq b$ (potentially empty) linking constraints. In the context of Dantzig-Wolfe decomposition, 
\eqref{original_MIP} is sometimes called the compact formulation. 
MIPs of this form 
with 
disjoint index sets $I(j)$ are called loosely coupled MIPs \cite{bodur2022decomposition}.
Note that we do not assume the original MIP \eqref{original_MIP} is loosely coupled and as such the supports of different blocks can have overlaps (e.g., in MIPs with two-stage \cite{ahmed2010two} or overlapping \cite{bartlett2005temporal} block structures).
We call the LP relaxation of \eqref{original_MIP}, obtained by dropping the integrality constraints $x\in X$, the natural LP relaxation, and denote its optimal objective value by $z_L$. 
Throughout, we assume that all data is rational.

\subsection{Dantzig-Wolfe Decomposition}\label{subsec:DWdecomp}
We next consider replacing the constraints $x_{I(j)}\in P^j$ in \eqref{original_MIP} by $x_{I(j)}\in \conv(Q^j)$, where 
$$ Q^j=\{y\in\bR^{|I(j)|}:G^jy\geq g^j,y\in X^j \},$$
and  $X^j$ has the integrality constraints inherited from $X$ for the variables $x_{I(j)}$.
Then, one obtains the  DW reformulation of problem \eqref{original_MIP}.
Relaxing the integrality constraints $x\in X$ in this formulation leads to the DW relaxation of \eqref{original_MIP}:
\begin{equation}\label{DW_reform}
	\begin{aligned}
		z_D:=\min\ &c^\top x\\
		\text{s.t. }&x_{I(j)}\in \conv(Q^j), \quad j\in J,\\
		&Ax\geq b.
	\end{aligned}
\end{equation}
The DW bound $z_D$ is potentially stronger than the natural LP relaxation bound $z_L$ as $\conv(Q^j)\subseteq P^j$ for all $j\in\{1,\ldots,q\}$. 
Consequently, we have $z^* \ge  z_{D}\ge z_{L}.$
In practice, computing $z_D$ is not a straightforward task as 
polyhedra $\big(\conv(Q^j)\big)_{j=1}^q$ are not given explicitly. 
For $j\in J$, let $V^j$ and $R^j$ denote the set of extreme points and the set of extreme rays of $\conv(Q^j)$, respectively.
The DW relaxation  \eqref{DW_reform}  can be reformulated using $(V^j)_{j\in J}$ and $(R^j)_{j\in J},$ leading to the following extended formulation: 
\begin{subequations}\label{DW_relaxation}
	\begin{alignat}{2}
		z_D=\min\ &c^\top x\\
		\text{s.t. }&x_{I(j)}=\sum_{v\in V^j}\lambda_vv+\sum_{r\in R^j}\mu_rr,~~ &j\in J,\label{eqn:xI}\\
		&\sum_{v\in V^j}\lambda_v=1,&j\in J,\label{DW_relaxation_c}\\
		&\lambda\geq 0,~\mu\geq 0,\\
		&Ax\geq b.
	\end{alignat}
\end{subequations}	

This formulation can now be solved iteratively via column generation. 
Specifically, at each iteration, a {\em restricted} LP is solved by replacing $V^j$ and $R^j$ with a small collection of extreme points $\hat{V}^j\subseteq V^j$ and extreme rays $\hat{R}^j\subseteq R^j$. In each iteration, a new extreme point or a new extreme ray  is generated by solving a {\em pricing} subproblem 
\begin{equation}\label{DWpricing}
	D_j(\pi^{j}):=\min\Big\{(\pi^{j})^\top v:v\in \conv(Q^j) \Big\}=\min\Big\{(\pi^{j})^\top v:v\in Q^j \Big\}
\end{equation}
for each block $j\in J$, where $\pi^{j}$ above is the dual solution associated with constraints \eqref{eqn:xI} in the restricted LP relaxation of \eqref{DW_relaxation}. By convention, we define $D_j(\pi^j)=-\infty$ if the pricing subproblem \eqref{DWpricing} is unbounded, and when  this happens one can obtain an extreme ray $r\in R^j$ by finding an unbounded ray of $\min\{(\pi^j)^\top v:v\in P^j\}$, which is added to $\hat R^j$. 
If, on the other hand, $D_j(\pi^j)$ is finite, one obtains a solution of \eqref{DWpricing} as an extreme point $v\in V^j$. This point is added  to $\hat V^j$ provided that it has a negative reduced cost that is computed by subtracting the dual variable associated with the $j$-th constraint of \eqref{DW_relaxation_c} from $D_j(\pi^{j})$.
The restricted LP, with augmented vertices  and rays, is then solved again and this process is repeated until no such points or rays are generated.

The restricted LP at each iteration gives an {\em upper bound} for $z_D$, and these upper bounds converge to $z_D$ in a finite number of iterations as $|V^j|$ and $|R^j|$ are finite for all $j\in J$. 
When the algorithm terminates, in addition to the lower bound $z_D$, an optimal solution to  \eqref{DW_reform} is obtained. 
If this solution does not  satisfy the integrality constraints $x\in X$, branching is necessary to obtain an optimal solution of the original problem. 
The subproblems in the branch-and-bound tree can again be solved using column generation, leading to a 
branch-and-price procedure.

\subsection{Dantzig-Wolfe Block Cuts}\label{subsec:DWB_cuts}\label{subsec:InnerAndOuter}
Note that while solving the DW relaxation, the pricing subproblems \eqref{DWpricing} can also be used to derive valid inequalities for \eqref{DW_reform}. 
Specifically, for each $j\in\{1,\ldots,q\}$ any inequality of the form $\pi^\top y\geq D_j(\pi)$ is valid for $\conv(Q^j)$   for   $\pi\in\bR^{I(j)}$. 
In this paper, we call these inequalities DW Block cuts.

%

\begin{defn}
	An inequality is called a \emph{Dantzig-Wolfe Block (DWB) cut} for \eqref{original_MIP} if it is of the form\begin{equation}\label{DWcuts}
		\pi^\top x_{I(j)}\geq D_j(\pi)
	\end{equation}
	for some $j\in\{1,\ldots,q\}$ and  $\pi\in\bR^{I(j)}$, where $D_j(\cdot)$ is defined in \eqref{DWpricing}.
\end{defn}
As DWB cuts are essentially valid inequalities for the block polyhedra $\big(\conv(Q^j)\big)_{j=1}^q$, they can be viewed as special cases of a broader class of cutting planes called Fenchel cuts \cite{boyd1994fenchel}.


Adding some of these DWB cuts to the  natural LP relaxation of \eqref{original_MIP} can lead to a stronger formulation and therefore a better dual bound than $z_L$.
Let $S$ denote the feasible region of the original problem \eqref{original_MIP}. We next present a relationship between the dimension of a DWB cut in $\conv(S)$ and its restriction in $\conv(Q^j)$.

\begin{defn}
	We say that MIP \eqref{original_MIP} has \emph{block relative feasibility} if, for each $j\in\{1,\ldots,q\}$ and $y\in Q^j$, there exists $x\in S$ such that $x_{I(j)}=y$, i.e., $\proj_{x_{I(j)}}(S)=Q^j$ for $j\in J$.
\end{defn}
The block relative feasibility assumption holds for a broad class of MIP problems with decomposable structures, including two-stage stochastic integer programs with relatively complete recourse \cite{shapiro2021lectures}.
\begin{prop}\label{prop:dim_of_face}
	Assume problem \eqref{original_MIP} has block relative feasibility. If $(\pi^j)^\top y\geq D_j(\pi^j)$ defines a $d$-dimensional face of $\conv(Q^j)$ for some $j\in\{1,\ldots,q\}$, then \eqref{DWcuts} defines a face of $\conv(S)$ of dimension at least $d$. 
\end{prop}
\begin{proof}
Since $(\pi^j)^\top y\geq D_j(\pi^j)$ defines a $d$-dimensional face of $\conv(Q^j)$, there exists $d+1$ affinely independent points $\{y^k\}_{k=1}^{d+1}\subseteq Q^j$. By the block relative feasibility assumption, there exist points $\{x^k\}_{k=1}^{d+1}\subseteq S$ such that $x^k_{I(j)}=y^k$. Points $\{x^k\}_{k=1}^{d+1}$ are affinely independent from each other by affine independence of $\{y^k\}_{k=1}^{d+1}$. The conclusion then follows from the fact that $\{x^k\}_{k=1}^{d+1}$ are all on the face associated with \eqref{DWcuts} in $\conv(S)$.
\end{proof}
Proposition \ref{prop:dim_of_face} indicates that DWB cuts whose restrictions in the space of $x_{I(j)}$ correspond to high-dimensional faces of $\conv(Q^j)$ are likely to define high-dimensional faces in $\conv(S)$. This motivates the idea of strengthening some DWB cuts to obtain higher-dimensional DWB cuts, which will be discussed in Section \ref{sec:Strengthening}. We next investigate how to generate critical DWB cuts for obtaining strong dual bounds.

\section{Incorporating the DW Bound Into the Formulation}\label{sec:DWintoBnC}


In a number of applications, it has been shown that the DW bound can be significantly stronger than the natural LP relaxation bound  of \eqref{original_MIP} \cite{savelsbergh1997branch,ceselli2005branch,caprara2013uncommon}. 
However, even if this bound 
can be  effectively computed, enforcing the integrality constraints $x\in X$ requires a specialized 
branch-and-price algorithm. A straightforward approach to recover the DW bound $z_D$ in the original space is to  add  a single cut  $c^\top x\geq z_D,$ to the LP relaxation of \eqref{original_MIP} which we call the \emph{objective function cut}. However, it is well known that adding such an objective function cut often slows down MIP solvers in practice. 

We next observe a basic property of the optimal face of the LP relaxation after adding the objective function cut.
\begin{prop}\label{prop:obj_cut}
	Let $P$ be a polyhedron in $\bR^n$. If neither $c^\top x\leq v$ nor $c^\top x\geq v$ is valid for $P$, then $\dim(P\cap\{x:c^\top x=v\})=\text{dim}(P)-1$.
\end{prop}
\begin{proof}
Define $P^+:=P\cap\{x:c^\top x\leq v\}$. Note that $c^\top x\leq v$ is an irredundant inequality for $P^+$ because $c^\top x\leq v$ is not valid for $P$. By \cite[Lemma 3.26]{conforti2014integer}, $\dim(P\cap\{x:c^\top x=v\})=\dim(P^+)-1$. We next show that the affine hull of $P^+$ is equal to the affine hull of $P$, and thus $\dim(P^+)=\dim(P)$. By \cite[Theorem 3.17]{conforti2014integer}, we only need to show the following two statements:\begin{enumerate}
	\item Equality $c^\top x=v$ is not valid for $P^+$;
	\item If $a^\top x\leq a_0$ is valid for $P$ but $a^\top x=a_0$ is not valid for $P$, then $a^\top x=a_0$ is not valid for $P^+$.
\end{enumerate} 
Note that $c^\top x\geq v$ is not valid for $P$. Therefore, there exists $\hat{x}\in P$ such that $c^\top \hat{x}<v$. The first statement then follows from the fact that $\hat{x}\in P^+$. Similarly, there exists $\bar{x}\in P$ such that $a^\top \bar{x}< a_0$. For $\lambda\in(0,1)$, define $x^{\lambda}:=(1-\lambda)\hat{x}+\lambda\bar{x}$. Because $a^\top \bar{x}\leq a_0$, we have $x^{\lambda}\in P^+$ but $a^\top x^{\lambda}<a_0$ for a small enough $\lambda$. The second statement then follows.
\end{proof}

Therefore, if $z_D>z_L$, adding the objective function cut to the original formulation would often create an optimal face almost as high-dimensional as the original LP relaxation polyhedron.
This, in turn, can cause not only performance variability \cite{lodi2013performance}, but also serious computational issues especially in early stages of the branch and cut in terms of branching \cite{gamrath2020exploratory}, as well as cutting plane generation.

\subsection{An alternative approach}
In Section \ref{subsec:InnerAndOuter}, we observed that valid inequalities for the DW relaxation can be generated while solving the pricing subproblems.
We next show that using a small number of such cuts 
can readily recover $z_D$. 
Assume that at iteration $t$ of DW decomposition, the restricted LP is of the form
\begin{equation}\label{eqn:restrictedLP}
	\begin{alignedat}{3}
		{z}_D^t=\min\ & c^\top x\\
		\text{s.t. }&x_{I(j)}=\sum_{v\in \hat{V}^j}v\lambda^j_v+\sum_{r\in\hat{R}^j}r\mu_r^j,\quad &&j\in J,\quad&&(\pi^{j,t})\\
		&Ax\geq b,\quad&&&&(\beta^t)\\
		&\sum_{v\in\hat{V}^j }\lambda^j_v=1,\quad &&j\in J,\quad&&(\theta_{j}^{t})\\
		&\lambda^j\geq 0,~\mu^j\geq 0,\quad&& j\in J,
	\end{alignedat}
\end{equation}
and let $(\pi^{1,t},\ldots,\pi^{q,t},\beta^t,\theta^{t}_1,\ldots,\theta^{t}_q)$ be the optimal dual solution 
associated with this restricted LP. We then have the following result for DWB cuts derived from the optimal dual solution of the last iteration of DW decomposition.
\begin{thm}\label{thm:DWBlast_iter}
	Assume DW decomposition terminates in $\T$ iterations. Then,\begin{subequations}\label{DWBLP}
		\begin{align}
			z_D=\min\ &c^\top x\label{DWBobj}\\
			\text{s.t. }&(\pi^{j,\T})^\top x_{I(j)}\geq D_j(\pi^{j,\T}),\quad j\in J,\label{DWB:lastIter}\\
			&Ax\geq b.\label{con:linking}
		\end{align}
	\end{subequations}
\end{thm}
\begin{proof}
The ``$\geq$" direction of \eqref{DWBobj} is implied by the definition of $z_D$ in \eqref{DW_reform} as inequality \eqref{DWB:lastIter} is valid for $\conv(Q^j)$ for $j\in J$. We next show the ``$\leq$" direction.
Based on LP duality of \eqref{eqn:restrictedLP} at iteration $\T$ and the termination condition of DW decomposition, the following equalities hold:\begin{enumerate}
	\item $z_D=b^\top \beta^{\T}+\sum_{j=1}^q\theta^{\T}_j$;
	\item $c_i=A_i^\top \beta^{\T}+\sum_{j:i\in I(j)}\pi_i^{j,\T}$, $i=1,\ldots,n$.
\end{enumerate}
Note that at the last iteration $\T$, the DW pricing subproblems are bounded. Let $(v^{j,\T})_{j=1}^q$ denote the solutions of the DW pricing subproblems at iteration $\T$.
Note that the reduced costs associated with points $(v^{j,\T})_{j=1}^q$ are nonnegative at iteration $\T$ of DW decomposition, i.e., $(\pi^{j,\T})^\top v^{j,\T}-\theta^{\T}_j=D_j(\pi^{j,\T})-\theta^{\T}_j\geq 0$ for $j\in J$. Therefore, for each solution $x$ satisfying \eqref{DWB:lastIter} and \eqref{con:linking}, we have the following inequality:\begin{multline}\label{eqn:LIDWB_dualbound}
	c^\top x=\sum_{i=1}^nc_ix_i=\sum_{i=1}^n\Big[x_iA_i^\top \beta^{\T}+\sum_{j:i\in I(j)}x_i\pi_i^{j,\T}\Big]={\underbrace{(\beta^{\T})}_{\geq 0}}^\top \underbrace{A x}_{\geq b}+\sum_{j=1}^q\underbrace{(\pi^{j,\T})^\top x_{I(j)}}_{\geq D_j(\pi^{j,\T})}\\
	\geq b^\top \beta^{\T}+\sum_{j=1}^q D_j(\pi^{j,\T})\geq b^\top \beta^{\T}+\sum_{j=1}^q\theta^{\T}_j=z_D.
\end{multline}
\end{proof}
We call inequalities \eqref{DWB:lastIter} \emph{last-iteration DWB cuts}. Theorem \ref{thm:DWBlast_iter} shows that $q$ last-iteration DWB cuts together with linking constraints recover the DW bound $z_D$. We remark that the last-iteration DWB cuts are not necessarily all nontrivial. It is possible that $\pi^{j,\T}=0$ for some $j\in J$, which implies that the convexification of the $j$-th block has no impact on improving the dual bound.

It is also worth emphasizing that \eqref{DWBLP} is not a valid formulation for the MIP \eqref{original_MIP} even if we add integrality constraints $x\in X$ to it. 
One should use last-iteration DWB cuts as cutting planes and add them to the original formulation \eqref{original_MIP} to obtain a valid formulation whose LP relaxation bound is precisely $z_D$.
Also note that Theorem \ref{thm:DWBlast_iter} does not imply that last-iteration DWB cuts dominate other DWB cuts that can be generated at intermediate iterations $\tau<\T$, in the sense that intermediate-iteration DW cuts may still cut off fractional points that do not violate any of the last-iteration DWB cuts.

\subsection{Dual Degeneracy and LP Optimal Face}\label{sec:dual_deg}
When comparing the strength of different collections of cutting planes or different formulations, very often the LP relaxation bound is used as the sole criterion. However, the effectiveness of two formulations in branch and cut may differ significantly even when they have very similar (or, the same) LP relaxation bounds.  An additional property that should also be taken into account is the dual degeneracy  of the LP relaxation of the formulation \cite{gamrath2020exploratory}.
A dual basic solution of an LP is called \emph{dual degenerate} if at least one of the dual basic variables is set to 0 in that solution.
Next, we formally define the degeneracy level of a dual basic solution of an LP (given in inequality form).

\begin{defn} Consider an LP with $n$ variables and $m$ inequality constraints, and let $\myalpha$ be a basic feasible {\em dual} solution.
	We define the \emph{degeneracy level} of $\myalpha$ to be $n-\|\myalpha\|_0$.
\end{defn}

A highly dual degenerate LP relaxation is associated with many alternative LP basic primal optimal solutions, which usually corresponds to a large optimal face. The following result shows how the size of the optimal face (more precisely, its dimension) is related to the degeneracy level of a dual basic optimal solution.
\begin{prop}\label{prop:opt_face}
	Assume $\myalpha^*\in\bR^m_+$ is a dual basic optimal solution of an LP with $n$ variables and $m$ inequality constraints. Then, the optimal face of the LP has dimension at most $n-\|\myalpha^*\|_0$.
	Furthermore, if $\myalpha^*$ is the unique dual optimal solution, then the optimal face of the LP has dimension exactly $n-\|\myalpha^*\|_0$.
\end{prop}
\begin{proof}
Assume the LP is of the form $\min\{c^\top x:Gx\geq h\}$.
Let $F$ denote the optimal face of the LP, i.e.,\begin{equation}\label{eqn:optimal_face}
	F=\{x:Gx\geq h,~c^\top x\leq h^\top \myalpha^*\}.
\end{equation}
Let $(g^k)^\top $ denote the $k$-th row of $G$. By complementary slackness of LP, $(g^k)^\top x=h_k$ for all $x\in F$ for all $k$ with $\myalpha^*_k> 0$. Since $\myalpha^*$ is a dual basic optimal solution, $\{(g^k,h_k)\}_{k:\myalpha^*_k> 0}$ are linearly independent. Otherwise, there exists $\beta\in\bR^m\setminus\{\mathbf{0}\}$ such that $\beta_k=0$ for all $k$ with $\myalpha_k^*=0$ and $\sum_{k:\myalpha_k^*> 0}\beta_k(g^k,h_k)=\mathbf{0}$. Then, note that $\myalpha^*+\epsilon\beta$ and $\myalpha^*-\epsilon\beta$ are both dual optimal solutions of the LP for small enough positive $\epsilon$, which contradicts the fact that $\myalpha^*$ is a dual basic optimal solution.
Therefore, $\dim(F)\leq n-\text{rank}(\{g^k\}_{k:\myalpha^*_k> 0})=n-\text{rank}(\{(g^k,h_k)\}_{k:\myalpha^*_k> 0})=n-\|\myalpha^*\|_0$. Here, the first inequality follows from \cite[Theorem 3.17]{conforti2014integer}, the second equality follows from the consistency of the linear system $\{(g^k)^\top x=h_k\}_{k:\myalpha^*_k> 0}$, and the third equality follows from linear independence of $\{(g^k,h_k)\}_{k:\myalpha^*_k> 0}$.

If $\myalpha^*$ is the unique dual optimal solution, by strict complementary slackness of LP \cite{goldman20164}, there exists an optimal solution $x^*\in F$ of the LP, such that $(g^k)^\top x^*>h_k$ for all $k$ with $\myalpha^*_k=0$. It implies that $\{(g^k)^\top x=h_k\}_{k:\myalpha^*_k> 0}$ and $c^\top x=h^\top \myalpha^*$ are exactly all the implicit equalities that hold in the inequality description \eqref{eqn:optimal_face} of $F$. By LP duality, $c^\top x=h^\top \myalpha^*$ is implied by $\{(g^k)^\top x=h_k\}_{k:\myalpha^*_k> 0}$. Therefore, by \cite[Theorem 3.17]{conforti2014integer}, $\dim(F)= n-\text{rank}(\{g^k\}_{k:\myalpha^*_k> 0})=n-\|\myalpha^*\|_0$.\end{proof}

We remark that Proposition \ref{prop:opt_face} does not extend to dual nonbasic optimal solutions, moreover, the $\ell_0$-norm of dual nonbasic optimal solutions can be greater than $n$. Note that the dual optimal solution is unique when the primal solution is nondegenerate. For the primal degenerate case, even if there exists a unique dual \emph{basic} optimal solution, it is possible that the dimension of the optimal face of the LP is strictly less than the degeneracy level of that dual basic optimal solution. See Example \ref{exmp:unique_primal} in Appendix \ref{apd:exmps} for an example where the unique dual \emph{basic} optimal solution has a strictly positive dual degeneracy level but the primal optimal solution is still unique.

Under some mild assumptions, Proposition \ref{prop:opt_face} implies the following.
\begin{prop}\label{Prop:DWBace_dim}
	Assume the LP \eqref{DWBLP} has a unique dual optimal solution. Then, the optimal face of \eqref{DWBLP} has dimension $n-q-\|\beta^{\T}\|_0$.
\end{prop}
\begin{proof}
Note that the proof of Theorem \ref{thm:DWBlast_iter} implies that $(1,\ldots,1,\beta^{\T})$ is a dual optimal solution of \eqref{DWBLP}. The result then follows from Proposition \ref{prop:opt_face}.
\end{proof}
Proposition \ref{Prop:DWBace_dim} also implies that if the dual optimal solution is unique, then none of the last-iteration DWB cuts can be redundant. 
If this is not the case, the dimension of the optimal face after adding the last iteration cuts depends on the number of cuts that are active. Note that when applying the last-iteration DWB cuts in practice, we would add them to the original formulation \eqref{original_MIP}, resulting in an LP optimal face whose size can be even smaller due to constraints $x_{I(j)}\in P^j$, $j\in J$. 

\section{ Bound Computation and Cut Generation via Lagrangian Relaxation}\label{sec:Lag}

We next discuss an alternative approach to generate cutting planes to recover the DW bound that uses Lagrangian relaxation \cite{fisher1981lagrangian}.
As we discuss later, this approach has better computational performance in practice due to stabilization.

In  Lagrangian relaxation, separate auxiliary variables are created for each block and these auxiliary variables are  related to the original variables using additional (copying) constraints.  The copying constraints together with the linking constraints $Ax\ge b$ are then dualized into the objective to obtain a Lagrangian relaxation of \eqref{DW_reform}. 
More precisely, one writes
\begin{subequations}\label{equivalent_MIP}
	\begin{alignat}{3}
		z_D=\min\ &c^\top x\\
		\text{s.t. }&y^j\in \conv(Q^j),\quad &&j\in J,\\
		&y^j=x_{I(j)}, &&j\in J,\quad&&\qquad(\pi^j)\label{equivalent_MIP_C}\\
		&Ax\geq b.&&&&\qquad(\beta)\label{equivalent_MIP_D}
	\end{alignat}
\end{subequations}
After dualizing constraints \eqref{equivalent_MIP_C} and \eqref{equivalent_MIP_D} 
using multipliers $\pi$ and $\beta\ge0$, 
one obtains the following Lagrangian relaxation:\begin{equation}\label{Lag_relax}
	\begin{aligned}
		z(\pi,\beta)=\min\ &c^\top x+\sum_{j=1}^q (\pi^j)^\top (y^j-x_{I(j)})+\beta^\top (b-Ax),\\
		\text{s.t. }&y^j\in Q^j,\quad j\in J.
	\end{aligned}
\end{equation}
Note that when the index sets $\{I(j)\}_{j=1}^q$ are nonoverlapping, \eqref{equivalent_MIP} and \eqref{Lag_relax} can be simplified by properly removing the copying constraints and the associated dual variables $\pi$.

In general, it follows from Lagrangian duality \cite{wolsey1999integer} that the largest Lagrangian relaxation bound matches $z_D$, i.e.,
\begin{equation}\label{eq:zeros}z_D=\max_{\beta\ge0,\pi}\ z(\pi,\beta).\end{equation}
Note that  $x$ is unconstrained in \eqref{Lag_relax} and therefore $z(\pi,\beta)=-\infty$ unless the coefficients of the $x$ variables in the objective function  are zero, i.e.,
$$c_i-\sum_{j:i\in I(j)}\pi^j_i-\beta^\top A_i=0\qquad \text{ for all } i\in\{1,\ldots,n\},$$
where   $A_i$ denotes the $i$-th column of $A$. Consequently, the Lagrangian dual problem \eqref{eq:zeros} can be equivalently written as the following Wolfe dual problem:
\begin{subequations}\label{Lag_dual}
	\begin{align}
		z_D=\max\ &z(\pi,\beta)\label{dual_obj}\\
		\text{s.t. }&\sum_{j:i\in I(j)}\pi^j_i+\beta^\top A_i=c_i,\quad i=1,\ldots,n,\label{dual_con}\\
		&\beta\geq 0.\label{dual_sign}
	\end{align}
\end{subequations}
For $(\pi,\beta)$ satisfying \eqref{dual_con} and \eqref{dual_sign}, it holds that $$z(\pi,\beta)=\sum_{j=1}^q D_j(\pi^j)+b^\top \beta,$$where $D_j:\bR^{|I(j)|}\rightarrow \bR\cup\{-\infty\}$ is a piecewise linear concave function of the form \begin{equation}\label{Def:Dj}
	D_j(\pi^j)=\min\{(\pi^j)^\top v:v\in Q^j \}.
\end{equation}
Therefore, \eqref{Lag_dual} is a nonsmooth convex optimization problem with a separable objective function. It is worth emphasizing that the pricing problem \eqref{DWpricing} in DW decomposition has exactly the same form as \eqref{Def:Dj}. The function values and supergradients of the concave function $D_j(\cdot)$ can be evaluated by solving \eqref{Def:Dj} \cite{fisher1981lagrangian} (an optimal solution of \eqref{Def:Dj} is a supergradient of $D_j(\cdot)$ at $\pi^j$).
This alternative way of viewing DW bound $z_D$ as the optimal value of \eqref{Lag_dual} allows us to use various convex optimization methods for computing DW bound $z_D$. For example, DW decomposition is equivalent to applying the classical cutting plane method \cite{kelley1960cutting} to solve \eqref{Lag_dual}. Since the description of $Q^j$ involves integer variables in general, functions $D_j(\cdot)$ are often piecewise linear concave with exponentially many pieces. In that case, convex optimization methods with some stabilization techniques (e.g., the level method \cite{lemarechal1995new}) often outperform the cutting plane method, and the difference can be significant. Figure \ref{Fig:CutplVsLevel} is a representative example of the difference in performance between the cutting plane method and the level method for computing the DW bound $z_D$ (averaged over a set of multiple knapsack assignment problem instances that are used in Section \ref{sec:MKAP}). Details about our implementation of the level method are presented in Appendix \ref{apd:level}.

\begin{figure}[bt!]
	\centering
	\caption{Comparison of the Cutting Plane Method (Left) and the Level Method (Right)}
	\label{Fig:CutplVsLevel}
	\includegraphics[width=0.43\linewidth]{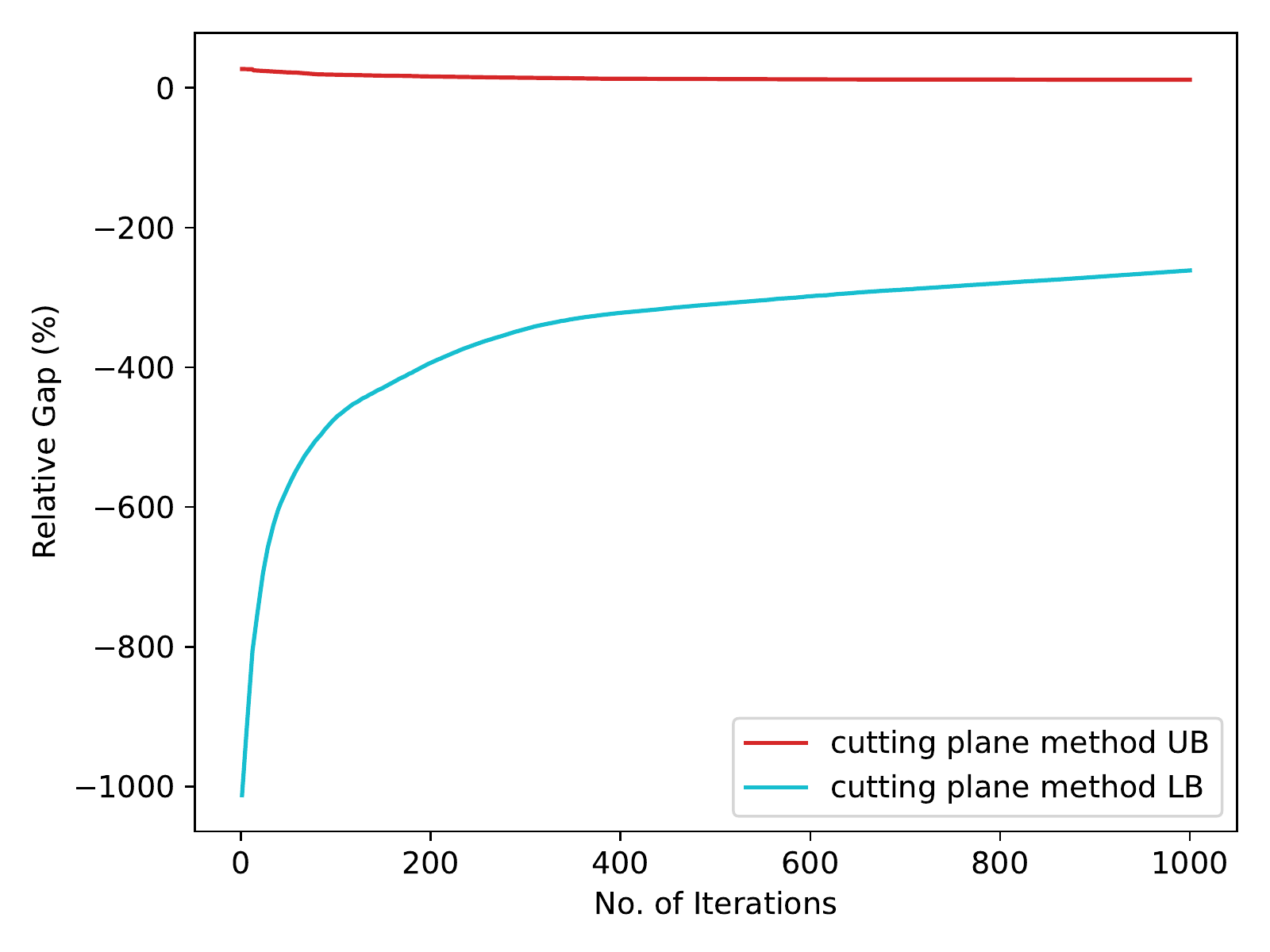}
	\includegraphics[width=0.43\linewidth]{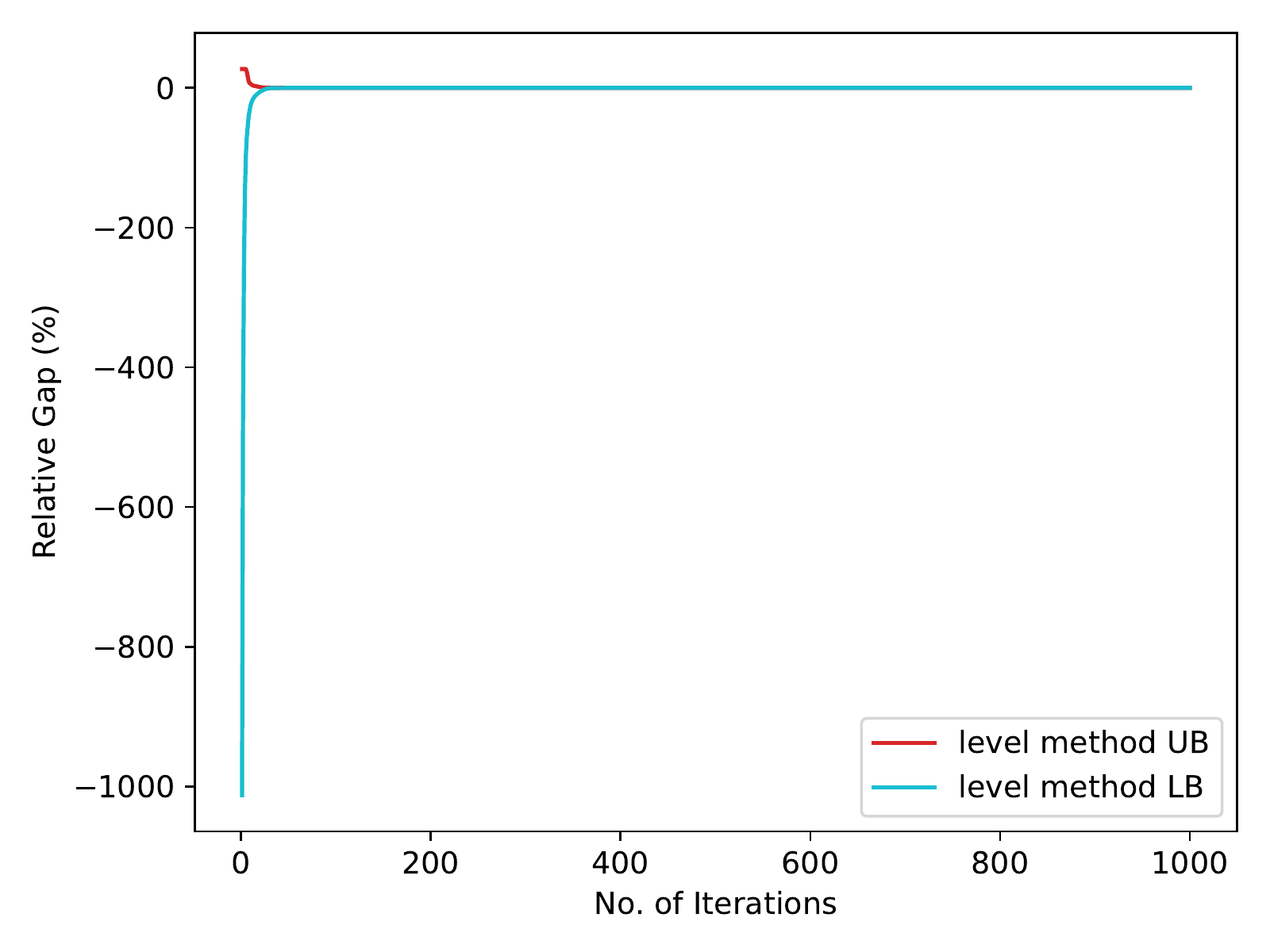}
\end{figure}

\subsection{Cut Generation From the Dual}
As discussed earlier, solving the dual problem can be computationally more efficient than the standard DW decomposition. During the solution of the dual problem \eqref{Lag_dual}, DWB cuts similar to \eqref{DWcuts} can also be generated every time we evaluate the function values of $D_j(\cdot)$. The following result demonstrates the strength of DWB cuts generated from the evaluation of the Lagrangian dual function at any point $(\pi,\beta)$ with $z(\pi,\beta)>-\infty$.
\begin{prop}\label{prop:dual_bnd}
	Let $(\pi,\beta)$ be dual multipliers for \eqref{Lag_relax} satisfying constraints \eqref{dual_con} and \eqref{dual_sign}.
	Then, \begin{equation}\label{DWBbnd:intermediate}\begin{aligned}
			z(\pi,\beta)\leq \min\ &c^\top x\\
			\text{s.t. }&(\pi^j)^\top x_{I(j)}\geq D_j(\pi^j),\quad j\in J,\\[.2cm]
			&Ax\geq b.
		\end{aligned}
	\end{equation}
\end{prop}

\begin{proof}
Consider any $x\in\bR^n$ feasible to the right-hand side LP of \eqref{DWBbnd:intermediate}. Since $c_i=\sum_{j:i\in I(j)}\pi^j_i+\beta^\top A_i$ for $i=1,\ldots,n$ by \eqref{dual_con}, we have\begin{multline*}
	c^\top x=\sum_{i=1}^n\sum_{j:i\in I(j)}\pi^j_ix_i+\beta^\top Ax=\sum_{j=1}^q\underbrace{(\pi^j)^\top x_{I(j)}}_{\geq D_j\big(\pi^j\big)}+{\underbrace{\beta}_{\geq 0}}^\top \underbrace{Ax}_{\geq b}
	\geq\sum_{j=1}^pD_j(\pi^j)+b^\top \beta(\tau)=z(\pi,\beta).
\end{multline*}
\end{proof}
Note that, unlike Theorem \ref{thm:DWBlast_iter}, Proposition \ref{prop:dual_bnd} does not depend on how the dual multiplier is obtained. If one can solve the dual problem \eqref{Lag_dual} to optimality, then Proposition \ref{prop:dual_bnd} implies that one can recover DW bound using DWB cuts associated with that optimal solution of \eqref{Lag_dual}. The single block case ($q=1$) of Proposition \ref{prop:dual_bnd} simplifies to the idea of Lagrangian cuts \cite{shapiro1971generalized}. 
\begin{cor}
	Let $(\bar{\pi},\bar{\beta})$ be an optimal solution of \eqref{Lag_dual}. Then,\begin{subequations}\label{DWoptdual_LP}\begin{align}
			z_D=\min\ &c^\top x\\
			\text{s.t. }&(\bar{\pi}^j)^\top x_{I(j)}\geq D_j(\bar{\pi}^j),\quad j\in J,\label{DWB:dual_lastIter}\\[.2cm]
			&Ax\geq b.
		\end{align}
	\end{subequations}
\end{cor}
Results similiar to Propostion \ref{Prop:DWBace_dim} can be derived for the optimal face of the LP \eqref{DWoptdual_LP} that utilizes DWB cuts associated with an optimal Lagrangian dual solution.
We also call inequalities \eqref{DWB:dual_lastIter} \emph{last-iteration DWB cuts} as they recover the same dual bound $z_D$ and are often generated in the last iteration of the Lagrangian dual algorithm.
Even if the Lagrangian dual problem is not solved to optimality, Proposition \ref{prop:dual_bnd} still guarantees that a dual bound that is at least as strong as the best Lagrangian dual bound can be obtained by generating DWB cuts associated with the dual solution that provides the strongest bound so far. Besides, there may be values for adding DWB cuts obtained at different dual multipliers. See Example \ref{exmp:intermediate_cuts} in Appendix \ref{apd:exmps} showing that DWB cuts can potentially provide stronger dual bounds than the best Lagrangian dual bound.

\section{Generating a Stronger Relaxation}\label{sec:Strengthening}
In this section we describe how to strengthen DWB cuts to obtain a stronger relaxation. The strengthened cutting planes are still DWB cuts and therefore do not lead to bounds stronger than $z_D$. However, these strengthened cutting planes may potentially make more constraints active in the LP relaxation, and therefore can help reduce the dual degeneracy level of the formulation and improve MIP solver performance. Moreover, the strengthened cutting planes are more likely to define high-dimensional faces of the original problem, as they define higher-dimensional faces of the block polyhedra $\conv(Q^j)$.

\subsection{Disjunctive Coefficient Strengthening}\label{sec:coef_str}
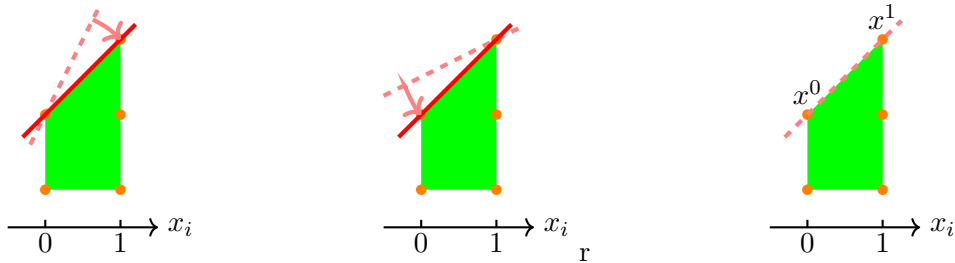
\begin{figure}[tb!]
	\centering
	\caption{Disjunctive Coefficient Strengthening for Three Cases (Dashed: Original Cut, Solid: Strengthened Cut)}
	\label{fig:str}
	\begin{subfigure}{0.3\textwidth}
		\centering
		\begin{tikzpicture}
			\fill[green] (1,1) -- (1,3) -- (0,2) -- (0,1) -- cycle;
			\fill[orange] (0,1) circle[radius = 0.07];
			\fill[orange] (0,2) circle[radius = 0.07];
			\fill[orange] (1,3) circle[radius = 0.07];
			\fill[orange] (1,1) circle[radius = 0.07];
			\fill[orange] (1,2) circle[radius = 0.07];
			\draw[red!50, dashed, ultra thick] (-0.2,1.6) -- (0.7,3.4);
			\draw[red, ultra thick] (-0.3,1.7) -- (1.2,3.2);
			\draw[red!50, ->, ultra thick] (0.632,3.264) to[out=-26.6,in=135] (1,3);
			\draw[black, ->, thick] (-0.5,0.5) to (1.5,0.5);
			\node at (1.8,0.5) {$x_i$};
			\draw[thick] (0,0.5) -- (0,0.6);
			\node at (0,0.3) {0};
			\draw[thick] (1,0.5) -- (1,0.6);
			\node at (1,0.3) {1};
		\end{tikzpicture}
	\end{subfigure}
	\begin{subfigure}{0.3\textwidth}
		\centering
		\begin{tikzpicture}
			\fill[green] (1,1) -- (1,3) -- (0,2) -- (0,1) -- cycle;
			\fill[orange] (0,1) circle[radius = 0.07];
			\fill[orange] (0,2) circle[radius = 0.07];
			\fill[orange] (1,3) circle[radius = 0.07];
			\fill[orange] (1,1) circle[radius = 0.07];
			\fill[orange] (1,2) circle[radius = 0.07];
			\draw[red!50, dashed, ultra thick] (-0.5,2.25) -- (1.3,3.15);
			\draw[red, ultra thick] (-0.3,1.7) -- (1.2,3.2);
			\draw[red!50, ->, ultra thick] (-0.2,2.35) to[out=116.36,in=135] (0,2);
			\draw[black, ->, thick] (-0.5,0.5) to (1.5,0.5);
			\node at (1.8,0.5) {$x_i$};
			\draw[thick] (0,0.5) -- (0,0.6);
			\node at (0,0.3) {0};
			\draw[thick] (1,0.5) -- (1,0.6);
			\node at (1,0.3) {1};
		\end{tikzpicture}r
	\end{subfigure}
	\begin{subfigure}{0.3\textwidth}
		\centering
		\begin{tikzpicture}
			\fill[green] (1,1) -- (1,3) -- (0,2) -- (0,1) -- cycle;
			\fill[orange] (0,1) circle[radius = 0.07];
			\fill[orange] (0,2) circle[radius = 0.07];
			\fill[orange] (1,3) circle[radius = 0.07];
			\fill[orange] (1,1) circle[radius = 0.07];
			\fill[orange] (1,2) circle[radius = 0.07];
			\draw[red!50, dashed, ultra thick] (-0.3,1.7) -- (1.25,3.25);
			\draw[black, ->, thick] (-0.5,0.5) to (1.5,0.5);
			\node at (1.8,0.5) {$x_i$};
			\draw[thick] (0,0.5) -- (0,0.6);
			\node at (0,0.3) {0};
			\draw[thick] (1,0.5) -- (1,0.6);
			\node at (1,0.3) {1};
			\node at (0,2.3) {$x^0$};
			\node at (1,3.3) {$x^1$};
		\end{tikzpicture}
	\end{subfigure}
\end{figure}
We first describe a disjunctive coefficient strengthening technique for binary variables \cite{andersen2010coefficient} that strengthens the coefficients of a valid inequality one at a time. 
Given a valid inequality $a^\top x\geq f$  for a mixed integer linear set $Q$, let  $Q^=:=\{x\in Q\::\ a^\top x=f\}$.
For a  binary variable $x_i$, its  coefficient $a_i$ in the cut can be strengthened if one of the following two cases hold: $(i)$  $x_i=0$ for all $x\in Q^=$, or,  $(ii)$ $x_i=1$ for all $x\in Q^=$.
If there are points $x^0,x^1\in Q^=$ such that $x^0_i=0$ and  $x^1_i=1$, then this approach does not improve (i.e., decrease) the coefficient of the variable.   Figure \ref{fig:str} shows two-dimensional examples of all three cases. 

Note that if we solve\begin{equation}\label{CT:subproblem0}
	\bar{f}=\min\{a^\top x:x\in Q,x_i=1\}
\end{equation}
and observe that  $\bar{f}>f$, then we can strengthen the original inequality $a^\top x\geq f$ to be $a^\top x\geq f+(\bar{f}-f)x_i$ using the disjunction
\begin{displaymath}
	Q=\{x\in Q:x_i=0\}\cup\{x\in Q:x_i=1\}.
\end{displaymath}
Similarly, if we solve \begin{equation}\label{CT:subproblem1}
	\bar{f}=\min\{a^\top x:x\in Q,x_i=0\},
\end{equation}
and observe that $\bar{f}>f$, then we can strengthen the original inequality $a^\top x\geq f$ to be $a^\top x\geq f+(\bar{f}-f)(1-x_i)$.
If either problem \eqref{CT:subproblem0} or \eqref{CT:subproblem1} is infeasible, then one can simply fix variable $x_i$ to 0 or 1, respectively.
It is easy to verify that the original inequality is implied by the strengthened inequality together with the bound constraint $x_i\geq 0$ or $x_i\leq 1$. Therefore, if the original inequality is active in the LP relaxation, then one round of coefficient strengthening (if applicable) would potentially make a bound constraint active in the LP relaxation and reduce the dual degeneracy level. Note that this approach does not increase the size of the  formulation. 

For strengthening a DWB cut obtained from a block $j$, we set $Q=Q^j$ and apply coefficient strengthening sequentially to all coefficients of binary variables. Note that using a different ordering of the binary variables may lead to different strengthened cutting planes in the end. For simplicity, we use the ordering of the variables in the original formulation to strengthen DWB cuts in our numerical experiments.
We also keep a set $L$ of points that are known to be elements of $Q^=$, generated from previous solutions of \eqref{Def:Dj}, \eqref{CT:subproblem0} and \eqref{CT:subproblem1}. If there are  $x^0,x^1\in L$ such that $x^0_i=0$ and $x_i^1=1$, then without solving  \eqref{CT:subproblem0} or \eqref{CT:subproblem1} we  conclude that 
disjunctive strengthening cannot be applied to the $i$-th coefficient. 

\subsection{Strengthening via Tilting}\label{subsec:tilting}
We next describe a tilting technique introduced by \cite{espinoza2010lifting,chvatal2013local} that starts with a valid inequality and iteratively tilts it to obtain a facet-defining inequality. 
Let $a^\top x\geq f$ be a  valid inequality for $Q$ and assume that it is not facet-defining. Also assume that $\conv(Q)$ is full dimensional and there is a set $Q'\subseteq Q$ such that all points $x\in Q'$  satisfy $a^\top x=f$. The algorithm first generates a point $\bar{x}\in Q\setminus Q'$ that satisfies $a^\top \bar{x}>f$, and a vector $(v,w)$ such that $v^\top x=w$ for all $x\in Q'\cup\{ \bar{x}\}$. 
The  algorithm then does the following:
\begin{enumerate}
	\item If $v^\top x\geq w$ is valid for $Q$, then the algorithm outputs $\bar{x}$. 
	Otherwise the algorithm computes the largest $\lambda^+\in\bR_+$ such that $(a+\lambda^+v)^\top x\geq f+\lambda^+w$ is valid for $Q$ and outputs this inequality together with a point $\bar{x}^+\in Q\setminus Q'$ satisfying $(a+\lambda^+v)^\top \bar{x}^+=f+\lambda^+w$;
	\item If $v^\top x\leq w$ is valid for $Q$, then the algorithm outputs $\bar{x}$.
	Otherwise the algorithm computes the largest $\lambda^-\in\bR_+$ such that $(a-\lambda^-v)^\top x\geq f-\lambda^-w$ is valid for $Q$ and outputs this inequality together with a point $\bar{x}^-\in Q\setminus Q'$ satisfying $(a-\lambda^-v)^\top \bar{x}^-\geq f-\lambda^-w$.
\end{enumerate}
Note that when $\conv(Q)$ is full dimensional, both  $v^\top x\geq w$  and $v^\top x\leq w$  cannot be valid and consequently we obtain two (possibly identical) valid inequalities whose conic combination implies the original inequality $a^\top x\geq f$ (depicted in Figure \ref{fig:tilt}). 
Moreover, each one of the these inequalities has one more \emph{known} feasible point on its associated face than $a^\top x\geq f$.
\cite{chvatal2013local} show that recursively tilting one of the two obtained valid inequalities leads to a facet-defining inequality for $\conv(Q)$.

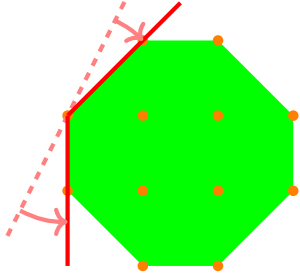
\begin{figure}[tb!]
	\centering
	\caption{Tilting One Valid Inequality Into Two Stronger Valid Inequalities}
	\label{fig:tilt}
	\begin{tikzpicture}
		\fill[green] (1,0) -- (2,0) -- (3,1) -- (3,2) -- (2,3) -- (1,3) -- (0,2) -- (0,1) -- cycle;
		\fill[orange] (0,1) circle[radius = 0.07];
		\fill[orange] (0,2) circle[radius = 0.07];
		\fill[orange] (1,0) circle[radius = 0.07];
		\fill[orange] (2,0) circle[radius = 0.07];
		\fill[orange] (3,1) circle[radius = 0.07];
		\fill[orange] (3,2) circle[radius = 0.07];
		\fill[orange] (1,3) circle[radius = 0.07];
		\fill[orange] (2,3) circle[radius = 0.07];
		\fill[orange] (1,1) circle[radius = 0.07];
		\fill[orange] (2,1) circle[radius = 0.07];
		\fill[orange] (1,2) circle[radius = 0.07];
		\fill[orange] (2,2) circle[radius = 0.07];
		\draw[red!50, dashed, ultra thick] (-0.8,0.4) -- (0.8,3.6);
		\draw[red, ultra thick] (0,2) -- (1.5,3.5);
		\draw[red, ultra thick] (0,2) -- (0,0);
		\draw[red!50, ->, ultra thick] (0.632,3.264) to[out=-26.6,in=135] (1,3);
		\draw[red!50, ->, ultra thick] (-0.632,0.736) to[out=-26.6,in=180] (0,0.586);
	\end{tikzpicture}
\end{figure}

In our context, we apply the tilting idea with the following modifications. For a DWB cut associated with block $j$, we set $Q$ to be $Q^j$,  and instead of picking one of the two tilted inequalities for the subsequent tilting iteration, we  apply  tilting  to both.   
By doing so, we create a binary tree where the root node corresponds to the original DWB cut and the remaining nodes correspond to inequalities obtained by tilting the inequalities associated with their parent nodes. 
For any such tree, cuts associated with the leaf nodes imply the original DWB cut associated with the root node.
We call the collection of inequalities associated with the leaf nodes of a depth-$d$  tree \emph{depth-$d$ tilted DWB cuts}. 
Using a depth-$d$ tree means replacing one DWB cut with up to $2^d$ DWB cuts, which can be computationally expensive if $d$ is large. Therefore, it is often beneficial to choose a relatively small value for $d$. 
Note that this process does not improve the DW bound but can potentially help reduce the dimension of the LP optimal face as more constraints are likely to become active at the optimal face.

To improve computational performance, we collect the feasible points for block $j\in J$ that we encounter during the overall algorithm and store them in a set $\hat{Q}^j\subseteq Q^j$.  
Using this set of  points can reduce the number of oracle calls needed as follows. 
First, we can check if any point in $\hat{Q}^j$ satisfies the condition for $\bar{x}$ and use it for choosing $(v,w)$, thus avoiding an extra call to the optimization oracle.
In addition, when computing $\lambda^+$, we use $\hat{Q}^j$ to obtain the following upper bound on  it:
\begin{displaymath}
	\lambda^+ := \min_{x\in Q^j:v^\top x<w}\frac{a^\top x-f}{w-v^\top x}~~~\leq ~~\min_{x\in \hat{Q}^j:v^\top x<w}\frac{a^\top x-f}{w-v^\top x}.
\end{displaymath}
This upper bound can be very close to the final $\lambda^+$ value and, in practice, we have observed a reduction in the number of oracle calls needed to compute $\lambda^+$. This in turn reduces the time spent on tilting significantly.
We use the same idea when computing $\lambda^-\in\bR_+$. 
Finally, we would like to remark that the overall procedure can be sensitive to the optimality gap tolerances in modern MIP solvers. One has to be cautious about the validity of the generated cuts in practice.


\section{Computational Experiments}\label{sec:MKAP}
We test our approaches on the Multiple Knapsack Assignment Problem (MKAP) introduced by \cite{kataoka2014upper} and the Temporal Knapsack Problem (TKP) as coined by \cite{bartlett2005temporal}. Detailed descriptions of MKAP and TKP, their DW reformulations as well as the test instances are presented in Appendices \ref{apd:MKAP} and \ref{apd:TKP}, respectively.

We compare the performance of the following formulations:
\begin{enumerate}
	\item \MIP: The original formulation \eqref{original_MIP}; \item \OBJ: Objective function cut $c^\top x\geq z_D$ added into the original formulation;
	\item \DWB: Last-iteration DWB cuts \eqref{DWB:dual_lastIter} added into the original formulation;
	\item \STR: Last-iteration DWB cuts \eqref{DWB:dual_lastIter} with disjunctive coefficient strengthening added into the original formulation;
	\item \DkT{d} : Last-iteration DWB cuts \eqref{DWB:dual_lastIter} with disjunctive coefficient strengthening and depth-$d$ tilting added into the original formulation.
\end{enumerate}

All experiments are run on a computer with a Intel(R) Xeon(R) Gold 6258R processor running at 2.7GHz, with up to 4 threads used. All optimization problems are solved using the optimization solver Gurobi with version 9.5.0 for MKAP and version 10.0.2 for TKP.

In MKAP, the parameters $|K|$, $|M|$, and $|N|$ define the size of the instances. In all the tables presented in this section, statistics presented in a row correspond to averages over 30 instances and 10 instances for MKAP and TKP, respectively. To avoid repetition, detailed statistics are reported only on MKAP test instances in Sections \ref{subsec:PolyRel}-\ref{subsec:BnC_MKAP}. For TKP, we concentrate on a smaller set of hard instances in Section \ref{sec:TKP}.

\subsection{Comparing Relaxations for MKAP}\label{subsec:PolyRel}
We first present some results regarding lower bounds for MKAP.  We compare the DW bound $z_D$ with the LP relaxation bound $z_L$ and the Gurobi root node bound (denoted by $z_R$) obtained by adding solver cuts. With some abuse of notation,  we use $z_D$ to denote  the best dual bound obtained by the level method. In theory, $z_D$ should be equal to the DW bound. However, numerical issues happen occasionally when solving the quadratic program in the level method, especially when the DW bound is close to the LP relaxation bound, in which case $z_D$ can be strictly less than $z_L$.

In Table \ref{Table:Comparison_DLR}, we present for MKAP the relative gaps $r_L:=(z_D-z_L)/|z_D|$ and $r_R:=(z_D-z_R)/|z_D|$ between $z_D$  and natural LP relaxation bound $z_L$ as well as the  Gurobi root bound $z_R$, respectively. The running time $t_R$ spent at the root node and running time $t_D$ spent on solving the Lagrangian dual problem by the level method are also reported.

\begin{table}[tb!]
	\scriptsize
	\centering
	\caption{Comparison of Different Bounds and Their Computing Time for MKAP}
	\label{Table:Comparison_DLR}
	\begin{tabular}{ccccccc}
		\toprule
		$|K|$ &$|M|$ &$|N|$ & $r_L$ ($\%$) & $r_R$ ($\%$) & $t_R$ (s) & $t_D$ (s)\\
		\midrule\midrule\addlinespace
		\phantom{0}2	&10	&\phantom{0}50	    &\phantom{-0}0.16	&\phantom{0}-0.04   &\phantom{0}0.1 &\phantom{0}0.8\\
		&	&100	&\phantom{-0}0.00	&\phantom{0}-0.03   &\phantom{0}0.3 &\phantom{0}2.0\\
		&	&200	&\phantom{-0}0.00	&\phantom{0}-0.01   &\phantom{0}0.5 &\phantom{0}9.9\\
		&	&300	&\phantom{-0}0.00	&\phantom{0}-0.01   & \phantom{0}0.5 &22.0\\
		&20	&\phantom{0}50	    &\phantom{-0}\textbf{1.98}	&\phantom{-0}0.08   &\phantom{0}0.4 &\phantom{0}0.7\\
		&	&100	&\phantom{-0}0.02	&\phantom{-0}0.00   &\phantom{0}0.3 &\phantom{0}2.8\\
		&	&200	&\phantom{-0}0.00	&\phantom{-0}0.00   &\phantom{0}0.6 &12.7\\
		&	&300	&\phantom{0}-0.01	&\phantom{0}-0.01   &\phantom{0}0.8 &37.0\\
		&30	&\phantom{0}50	    &\phantom{-}\textbf{12.05}	&\phantom{0}-0.06   &\phantom{0}0.1 &\phantom{0}0.6\\
		&	&100	&\phantom{-0}0.18	&\phantom{-0}0.10   &\phantom{0}0.4 &\phantom{0}2.8\\
		&	&200	&\phantom{-0}0.00	&\phantom{-0}0.00   &\phantom{0}0.8 &15.4\\
		&	&300	&\phantom{0}-0.05	&\phantom{0}-0.05   &\phantom{0}1.1 &49.7\\
		&40	&\phantom{0}50	    &\phantom{-}\textbf{39.11}	&\phantom{-0}0.00   &\phantom{0}0.0 &\phantom{0}0.7\\
		&	&100	&\phantom{-0}0.68	&\phantom{-0}0.18   &\phantom{0}1.9 & \phantom{0}2.7\\
		&	&200	&\phantom{-0}0.00	&\phantom{-0}0.00   &\phantom{0}0.9 &17.3\\
		&	&300	&\phantom{0}-0.34	&\phantom{0}-0.34   &\phantom{0}1.4 &55.3\\
		\phantom{0}5	&10	&\phantom{0}50	    &\phantom{-0}\textbf{1.27}	&\phantom{-0}0.44   &\phantom{0}0.3 &\phantom{0}0.6\\
		&	&100	&\phantom{-0}0.11	&\phantom{-0}0.07   &\phantom{0}0.2 &\phantom{0}2.4\\
		&	&200	&\phantom{-0}0.01	&\phantom{-0}0.00   &\phantom{0}0.4 &12.6\\
		&	&300	&\phantom{-0}0.00	&\phantom{-0}0.00   &\phantom{0}0.6 &43.0\\
		&20	&\phantom{0}50	    &\phantom{-0}\textbf{2.97}	&\phantom{-0}0.00   &\phantom{0}0.3 &\phantom{0}0.8\\
		&	&100	&\phantom{-0}0.10	&\phantom{-0}0.06   &\phantom{0}0.3 &\phantom{0}3.0\\
		&	&200	&\phantom{0}-0.44	&\phantom{0}-0.44   &\phantom{0}0.5    &12.9\\
		&	&300	&\phantom{0}-0.02	&\phantom{0}-0.02   &\phantom{0}0.8 &50.5\\
		&30	&\phantom{0}50	    &\phantom{-}\textbf{12.27}	&\phantom{-0}0.08   &\phantom{0}0.1 &\phantom{0}1.0\\
		&	&100	&\phantom{-0}0.41	&\phantom{-0}0.15   &\phantom{0}1.5 &\phantom{0}3.3\\
		&	&200	&\phantom{-0}0.01	&\phantom{-0}0.00   &\phantom{0}0.8 &17.7\\
		&	&300	&\phantom{-0}0.00	&\phantom{-0}0.00   &\phantom{0}1.2 &55.8\\
		&40	&\phantom{0}50	    &\phantom{-}\textbf{39.11}	&\phantom{-0}0.00 &\phantom{0}0.0 &\phantom{0}1.2\\
		&	&100	&\phantom{-0}1.00	&\phantom{-0}0.14 &\phantom{0}1.5 &\phantom{0}3.4\\
		&	&200	&\phantom{-0}0.01	&\phantom{-0}0.00   &\phantom{0}1.1 &23.1\\
		&	&300	&\phantom{0}-0.05	&\phantom{0}-0.05   &\phantom{0}1.5 &47.0\\
		\bottomrule
		
	\end{tabular}
	\quad
	\begin{tabular}{ccccccc}
		\toprule
		$|K|$ &$|M|$ &$|N|$ & $r_L$ ($\%$) & $r_R$ ($\%$) & $t_R$ (s) & $t_D$ (s)\\
		\midrule\midrule\addlinespace
		10 &10 &\phantom{0}50     &\phantom{-0}\textbf{7.03}   &\phantom{-0}0.86   &\phantom{0}0.2 &\phantom{0}0.5\\
		&	&100	&\phantom{-0}\textbf{5.65}  &\phantom{-0}\textbf{4.13}  &\phantom{0}0.2 &\phantom{0}1.2\\
		&	&200	&\phantom{-0}\textbf{3.53}  &\phantom{-0}\textbf{2.64}  &\phantom{0}0.3 &\phantom{0}4.5\\
		&	&300	&\phantom{-0}\textbf{3.64}  &\phantom{-0}\textbf{2.79}  &\phantom{0}0.5 &14.0\\
		&20 &\phantom{0}50     &\phantom{-0}\textbf{4.76}   &\phantom{0}-0.08   &\phantom{0}0.2 &\phantom{0}1.1\\
		&	&100	&\phantom{-0}0.74   &\phantom{-0}0.37   &\phantom{0}1.4 &\phantom{0}2.7\\
		&	&200	&\phantom{-0}0.02   &\phantom{-0}0.02   &\phantom{0}0.6 & 17.5\\
		&	&300	&\phantom{-0}0.00   &\phantom{-0}0.00   &\phantom{0}0.9 &45.5\\
		&30	&\phantom{0}50     &\phantom{-}\textbf{12.82}  &\phantom{-0}0.01   &\phantom{0}0.1 &\phantom{0}1.4\\
		&	&100	&\phantom{-0}0.88   &\phantom{-0}0.15   &\phantom{0}1.2 &\phantom{0}3.5\\
		&	&200	&\phantom{-0}0.02   &\phantom{-0}0.01   &\phantom{0}0.9 &22.8\\
		&	&300	&\phantom{-0}0.00   &\phantom{-0}0.00   &\phantom{0}1.3 &26.1\\
		&40	&\phantom{0}50	    &\phantom{-}\textbf{39.11}  &\phantom{-0}0.00   &\phantom{0}0.0 &\phantom{0}1.8\\
		&	&100	&\phantom{-0}\textbf{1.51}  &\phantom{-0}0.12   &\phantom{0}1.0 &\phantom{0}4.0\\
		&	&200	&\phantom{-0}0.04   &\phantom{-0}0.02   &\phantom{0}1.2 &26.6\\
		&	&300	&\phantom{-0}0.00   &\phantom{-0}0.00   &\phantom{0}1.8 &25.8\\
		25	&10	&\phantom{0}50     &\phantom{-}\textbf{43.05}  &\phantom{-0}0.21   &\phantom{0}0.1 &\phantom{0}0.9\\
		&	&100	&\phantom{-}\textbf{54.17} &\phantom{-0}\textbf{1.76}  &\phantom{0}0.3 &\phantom{0}1.2\\
		&	&200	&\phantom{-}\textbf{58.14} &\phantom{-}\textbf{13.64}  &\phantom{0}0.8 &\phantom{0}1.6\\
		&	&300	&\phantom{-}\textbf{66.02} &\phantom{-}\textbf{16.63}  &\phantom{0}1.5 &\phantom{0}2.1\\
		&20	&\phantom{0}50     &\phantom{-}\textbf{11.08}  &\phantom{0}-0.01   &\phantom{0}0.1 &\phantom{0}2.0\\
		&	&100	&\phantom{-0}\textbf{9.40}   &\phantom{-0}0.45   &\phantom{0}0.7 &\phantom{0}2.7\\
		&	&200	&\phantom{-0}\textbf{8.74}   &\phantom{-0}\textbf{6.83}  &\phantom{0}0.9 &\phantom{0}5.1\\
		&	&300	&\phantom{-}\textbf{10.18} &\phantom{-0}\textbf{8.78}  &\phantom{0}0.9 &\phantom{0}9.4\\
		&30	&\phantom{0}50     &\phantom{-}\textbf{14.35}  &\phantom{-0}0.00   &\phantom{0}0.0 &\phantom{0}2.8\\
		&	&100	&\phantom{-0}\textbf{3.80}   &\phantom{-0}0.13   &\phantom{0}1.0 &\phantom{0}4.2\\
		&	&200	&\phantom{-0}\textbf{1.79}   &\phantom{-0}0.85   &\phantom{0}7.0 &12.6\\
		&	&300	&\phantom{-0}\textbf{1.30}   &\phantom{-0}\textbf{1.29}  &\phantom{0}1.4 &28.3\\
		&40	&\phantom{0}50     &\phantom{-}\textbf{39.21}  &\phantom{-0}0.01   &\phantom{0}0.0 &\phantom{0}3.5\\
		&	&100	&\phantom{-0}\textbf{3.13}	&\phantom{-0}0.00   &\phantom{0}1.1 &\phantom{0}5.5\\
		&	&200	&\phantom{-0}0.75	&\phantom{-0}0.25   &11.5 &22.7\\
		&	&300	&\phantom{-0}0.25	&\phantom{-0}0.24   &\phantom{0}2.0 &24.8\\
		\bottomrule
	\end{tabular}
\end{table}

Depending on the instance size, the relative difference between the DW bound and the natural LP relaxation bound varies.
When the gap is small, Gurobi can often generate a root node bound that is almost as strong as or even slightly stronger than the DW bound. When the gap is large, Gurobi can close some gap at the root node but the bound is often significantly weaker than the DW bound. 

In our preliminary experiments, we observe that DWB cuts are more likely to be effective when DW bound is significantly stronger than the natural LP relaxation bound. Therefore, we focus on instance classes (defined by the size parameters $|K|$, $|M|$ and $|N|$) whose average natural LP relaxation bound is more than 1\% weaker than average DW bound, especially the ones with Gurobi root bound more than 1\% worse than DW bound. We denote these two sets of MKAP instances by $I_1$ and $I_2(\subseteq I_1)$, respectively. In Table \ref{Table:Comparison_DLR}, $I_1$ instances correspond to rows with bold $r_L$ and $I_2$ instances correspond to rows with bold $r_R$. Set $I_1$ consists of 870 instances and set $I_2$ consists of 270 instances. Over the $I_1$ instances, we observe that $z_D$ is very close to the DW bound with relative gaps always less than $0.01\%$ except on 3 instances (whose relative gaps are $0.03\%$, $0.16\%$ and $0.34\%$, respectively). 

We also observe that formulations with stronger strengthening tend to have less degenerate dual solutions. We report results regarding the dual degeneracy of optimal dual solutions associated with LP relaxations of different formulations in Appendix \ref{apd:dual_degen}. Time spent on generating these different formulations for MKAP is presented in Appendix \ref{apd:time_gen}.

\subsection{Objective Function Cut for MKAP}
We next emprically investigate how the objective function cut (formulation {\OBJ}) may increase the computing time on solving MKAP compared with simply using the original formulation {\MIP}.

\begin{figure}[bt!]
	\centering
	\caption{Relative Gaps Between Bounds and Optimal Value by {\MIP} (Left) and {\OBJ} (Right) for MKAP}
	\label{Fig:OBJ_LBUBs}
	\includegraphics[width=0.43\linewidth]{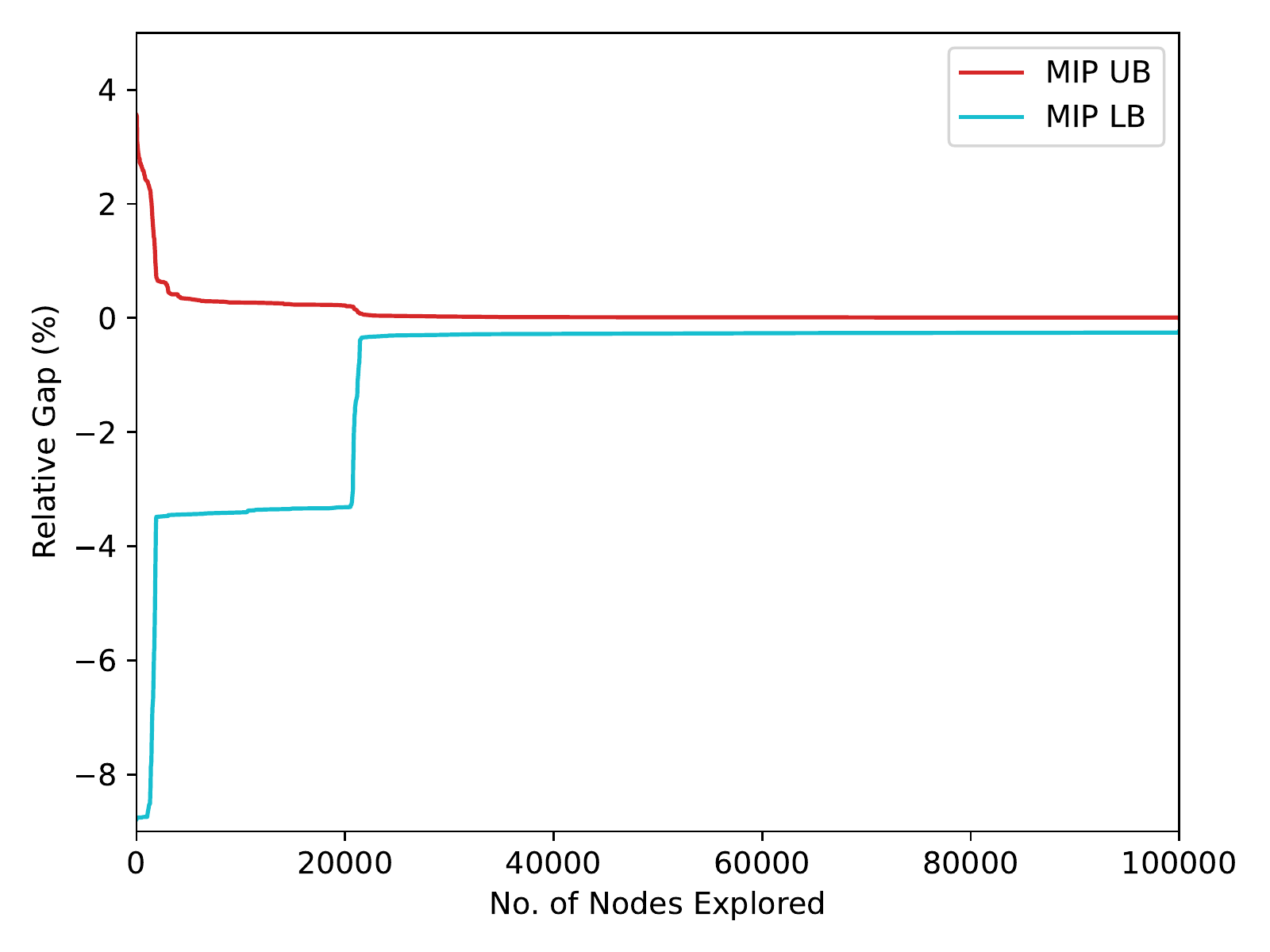}
	\includegraphics[width=0.43\linewidth]{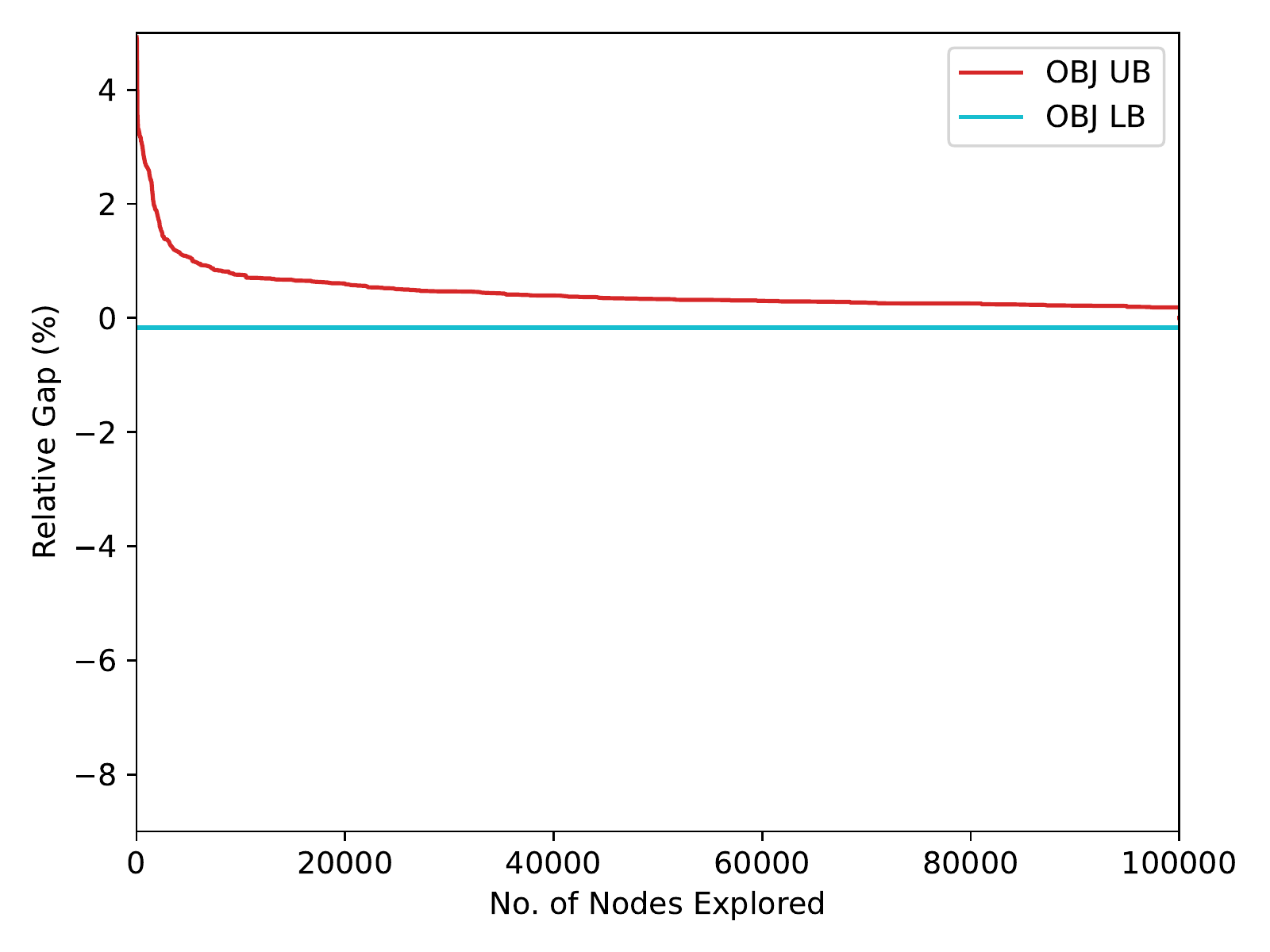}
\end{figure}

We consider MKAP instances with $(|K|,|M|,|N|)=(25,20,300)$, for which DW bound is much stronger than the natural LP relaxation bound (with average gap $10.18\%$). Out of these 30 instances, Gurobi can solve 22 using {\MIP} and 2 using {\OBJ}.
We plot in Figure \ref{Fig:OBJ_LBUBs} the average relative upper and lower bound gaps (with respect to the optimal value) obtained by Gurobi using formulations {\MIP} and {\OBJ} on $(25,20,300)$ instances as the number of branching nodes explored increases. We observe that although DW bound is almost equal to the optimal value for these instances (with a 0.01\% average relative gap), Gurobi can hardly produce good feasible solutions to improve the primal bound when using the {\OBJ} formulation. Instead, Gurobi can find better feasible solutions much more efficiently when using {\MIP}, where all optimal solutions are found after exploring a bit more than 20,000 nodes, leaving a few instances unsolved because Gurobi cannot prove optimality within the timelimit. This shows that a single objective function cut can already significantly influence the performance of Gurobi's primal heuristics, which is plausibly due to worse branching decisions and cut generation for {\OBJ}. A table reporting the average number of different cutting planes generated when using different formulations is presented in Appendix \ref{apd:nCuts}.

\subsection{MIP Experiments on MKAP}\label{subsec:BnC_MKAP}
\begin{table}[htb!]
	\centering
	\scriptsize
	\caption{Comparison of Gurobi Performance on MKAP $I_2$ Instances}
	\label{Table:BnC_I2}
	\begin{tabular}{crrrrrrrrrrrr}
		\toprule
		$(|K|,|M|,|N|)$ & \multicolumn{6}{c}{Number of Instances Solved} & \multicolumn{6}{c}{Average Solution Time (s)}\\
		\cmidrule(lr){2-7}\cmidrule(lr){8-13}
		& {\MIP} & {\OBJ} & {\DWB} & {\STR} & {\DkT{3} } & {\DkT{6} } & {\MIP} & {\OBJ} & {\DWB} & {\STR} & {\DkT{3} } & {\DkT{6} }\\
		\midrule\midrule\addlinespace
		(10,10,100) & 28/30 & 18/30 & \textbf{30}/30 & \textbf{30}/30 & \textbf{30}/30 & \textbf{30}/30 & $\geq$\phantom{0}65 & $\geq$298 & 2 & \textbf{0} & \textbf{0} & \textbf{0}\\
		(10,10,200) & 11/30 & 10/30 & \textbf{30}/30 & \textbf{30}/30 & \textbf{30}/30 & \textbf{30}/30 & $\geq$424 & $\geq$457 & 6 & \textbf{1} & 2 & 2\\
		(10,10,300) & 7/30 & 11/30 & \textbf{30}/30 & \textbf{30}/30 & \textbf{30}/30 & \textbf{30}/30 & $\geq$489 & $\geq$454 & 41 & \textbf{2} & 12 & 7\\
		(25,10,100) & \textbf{30}/30 & \textbf{30}/30 & \textbf{30}/30 & \textbf{30}/30 & \textbf{30}/30 & \textbf{30}/30 & \textbf{0} & 1 & \textbf{0} & \textbf{0} & \textbf{0} & \textbf{0}\\
		(25,10,200) & \textbf{30}/30 & \textbf{30}/30 & \textbf{30}/30 & \textbf{30}/30 & \textbf{30}/30 & \textbf{30}/30 & 2 & 25 & \textbf{0} & \textbf{0} & \textbf{0} & \textbf{0}\\
		(25,10,300) & \textbf{30}/30 & 29/30 & \textbf{30}/30 & \textbf{30}/30 & \textbf{30}/30 & \textbf{30}/30 & 6 & $\geq$\phantom{0}40 & \textbf{0} & \textbf{0} & 1 & 3\\
		(25,20,200) & 29/30 & 9/30 & \textbf{30}/30 & \textbf{30}/30 & \textbf{30}/30 & \textbf{30}/30 & $\geq$\phantom{0}44 & $\geq$499 & 3 & \textbf{2} & 1 & 1\\
		(25,20,300) & 22/30 & 2/30 & 29/30 & \textbf{30}/30 & \textbf{30}/30 & \textbf{30}/30 & $\geq$224 & $\geq$590 & $\geq$\phantom{0}48 & 5 & \textbf{3} & 7\\
		(25,30,300) & 1/30 & 0/30 & 0/30 & 2/30 & 1/30 & \textbf{3}/30 & $\geq$595 & $\geq$600 & $\geq$600 & $\geq$590 & $\geq$591 & $\geq$575\\
		\bottomrule
	\end{tabular}
\end{table}
\begin{table}[htb!]
	\centering
	\scriptsize
	\caption{Comparison of Gurobi Performance on MKAP $I_2$ Instances (Cont'd)}
	\label{Table:BnC_I2_condt}
	\begin{tabular}{crrrrrrrrrrrr}
		\toprule
		$(|K|,|M|,|N|)$ & \multicolumn{6}{c}{Average Optimality Gap (\%)} & \multicolumn{6}{c}{Average Number of Nodes}\\
		\cmidrule(lr){2-7}\cmidrule(lr){8-13}
		& {\MIP} & {\OBJ} & {\DWB} & {\STR} & {\DkT{3} } & {\DkT{6} } & {\MIP} & {\OBJ} & {\DWB} & {\STR} & {\DkT{3} } & {\DkT{6} }\\
		\midrule\midrule\addlinespace
		(10,10,100) & 0.04 & 0.03 & \textbf{0.00} & \textbf{0.00} & \textbf{0.00} & \textbf{0.00} & $\geq$\phantom{00}22761 & $\geq$1698156 & 4481 & 220 & 2 & \textbf{1}\\
		(10,10,200) & 0.27 & 0.09 & \textbf{0.00} & \textbf{0.00} & \textbf{0.00} & \textbf{0.00} & $\geq$2778880 & $\geq$3898948 & 6828 & 407 & 33 & \textbf{3}\\
		(10,10,300) & 0.21 & 0.09 & \textbf{0.00} & \textbf{0.00} & \textbf{0.00} & \textbf{0.00} & $\geq$2616971 & $\geq$2781174 & 131139 & 1208 & 631 & \textbf{6}\\
		(25,10,100) & \textbf{0.00} & \textbf{0.00} & \textbf{0.00} & \textbf{0.00} & \textbf{0.00} & \textbf{0.00} & 93 & 1123 & \textbf{1} & \textbf{1} & \textbf{1} & \textbf{1}\\
		(25,10,200) & \textbf{0.00} & \textbf{0.00} & \textbf{0.00} & \textbf{0.00} & \textbf{0.00} & \textbf{0.00} & 3853 & 20272 & \textbf{1} & \textbf{1} & \textbf{1} & \textbf{1}\\
		(25,10,300) & \textbf{0.00} & \textbf{0.00} & \textbf{0.00} & \textbf{0.00} & \textbf{0.00} & \textbf{0.00} & 8488 & $\geq$\phantom{00}25864 & \textbf{1} & \textbf{1} & \textbf{1} & \textbf{1}\\
		(25,20,200) & 0.02 & 0.08 & \textbf{0.00} & \textbf{0.00} & \textbf{0.00} & \textbf{0.00} & $\geq$\phantom{00}27987 & $\geq$\phantom{0}438310 & 2255 & 426 & 3 & \textbf{2}\\
		(25,20,300) & 0.24 & 0.17 & \textbf{0.00} & \textbf{0.00} & \textbf{0.00} & \textbf{0.00} & $\geq$\phantom{0}113353 & $\geq$\phantom{0}410728 & $\geq$\phantom{0}13861 & 929 & \textbf{1} & \textbf{1}\\
		(25,30,300) & 0.84 & 0.53 & 0.40 & 0.33 & 0.29 & \textbf{0.22} & $\geq$\phantom{00}81545 & $\geq$\phantom{0}379999 & $\geq$105115 & $\geq$35197 & $\geq$20036 & $\geq$6568\\
		\bottomrule
	\end{tabular}
\end{table}
We then conduct experiments on MKAP instances to compare Gurobi's performance on formulations {\MIP}, {\OBJ}, {\STR}, {\DkT{3} } and {\DkT{6} }. We set a 10-minute timelimit for each formulation. For $I_2$ instances, we summarize in Tables \ref{Table:BnC_I2} and \ref{Table:BnC_I2_condt} some statistics regarding the number of instances solved by different formulations, average solution time (excluding time for the Lagrangian dual problem and cut generation/strengthening), average ending optimality gap and average number of branching nodes needed to solve the instances. We observe that fewer instances are solved to optimality when using formulation {\OBJ} than formulation {\MIP}, which is consistent with the folklore that simply adding an objective function cut to improve dual bound is mostly ineffective. The ending optimality gap of {\OBJ} might be better than that of {\MIP} because {\OBJ} has a much stronger lower bound, whereas finding good feasible solutions can be hard for {\OBJ}. Without any strengthening, formulation {\DWB} already significantly outperforms both {\MIP} and {\OBJ}. For most instances, {\DkT{6} } requires the smallest number of branching nodes to solve the problem, while {\STR} has the best solution time. This can be explained by the fact that {\DkT{6} } has the strongest LP relaxation while {\STR} has a more compact formulation. Compared to {\DkT{3} } or {\DkT{6} }, the processing of each branching node takes less time for {\STR}. For the hardest $(25,30,300)$ instances, a stronger polyhedral relaxation becomes more important for closing the optimality gap, where {\DkT{6} } has the best performance. In Figure \ref{Fig:Compare_all} we plot gap closed as time or number of nodes explored increases for different methods on the hardest $(25,30,300)$ instances where the relative gap is computed by comparing with the best feasible solution found by all methods. We observe that although all methods except {\MIP} start from the same root bound (DW bound $z_D$), with coefficient strengthening and tilting the solver can find good primal solutions more efficiently. We also note that each branching node takes longer time to explore when coefficient strengthening and tilting are applied.
Additional tables comparing MIP solver performance on $I_1\setminus I_2$ instances are presented in Appendix \ref{apd:bnc_I2notI1}.

\begin{figure}[ht!]
	\centering
	\caption{Gap Closed as Time (Left) and Number of Nodes Explored (Right) Increases on a set of MKAP instances}
	\label{Fig:Compare_all}
	\includegraphics[width=0.43\linewidth]{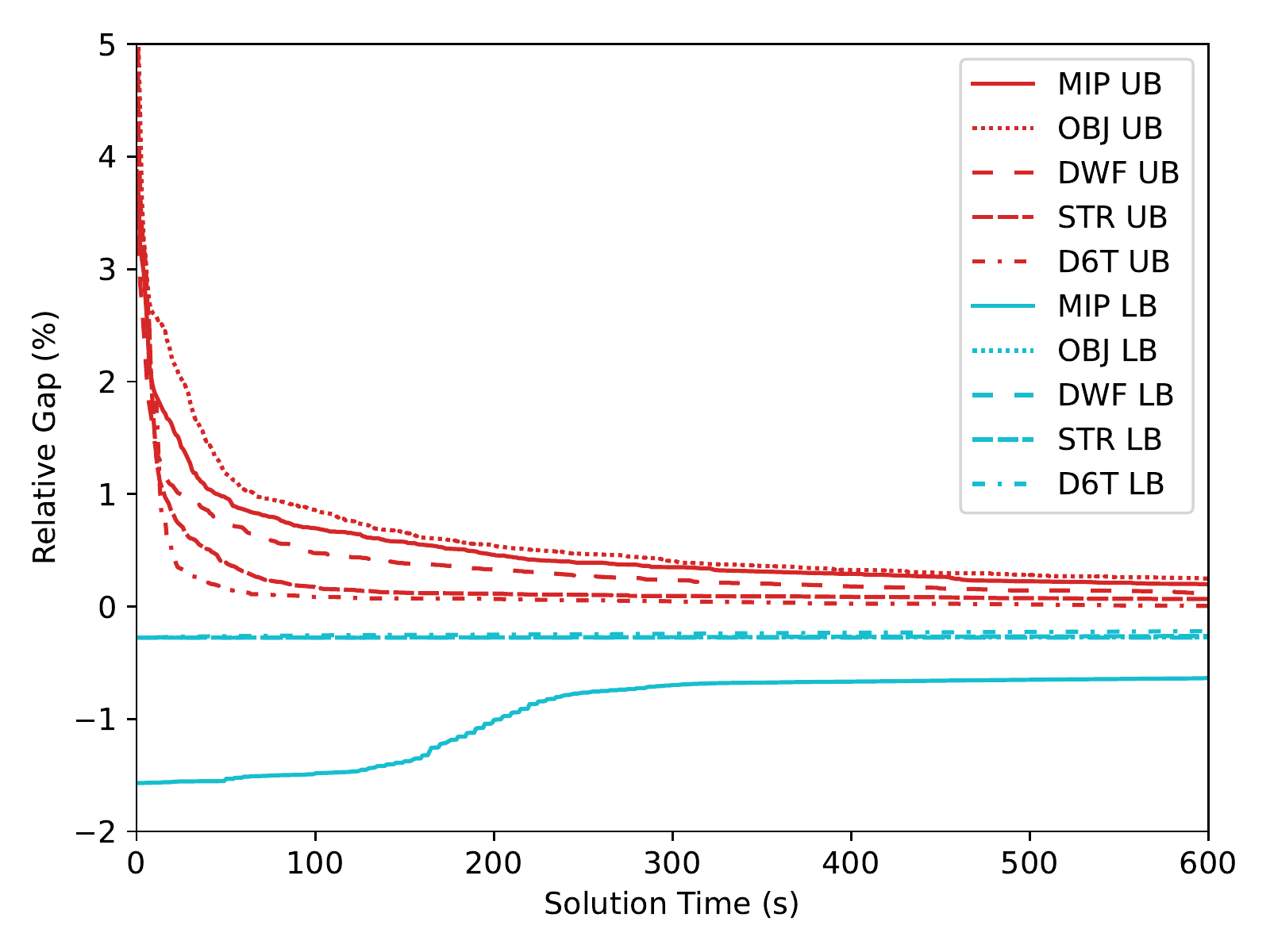}
	\includegraphics[width=0.43\linewidth]{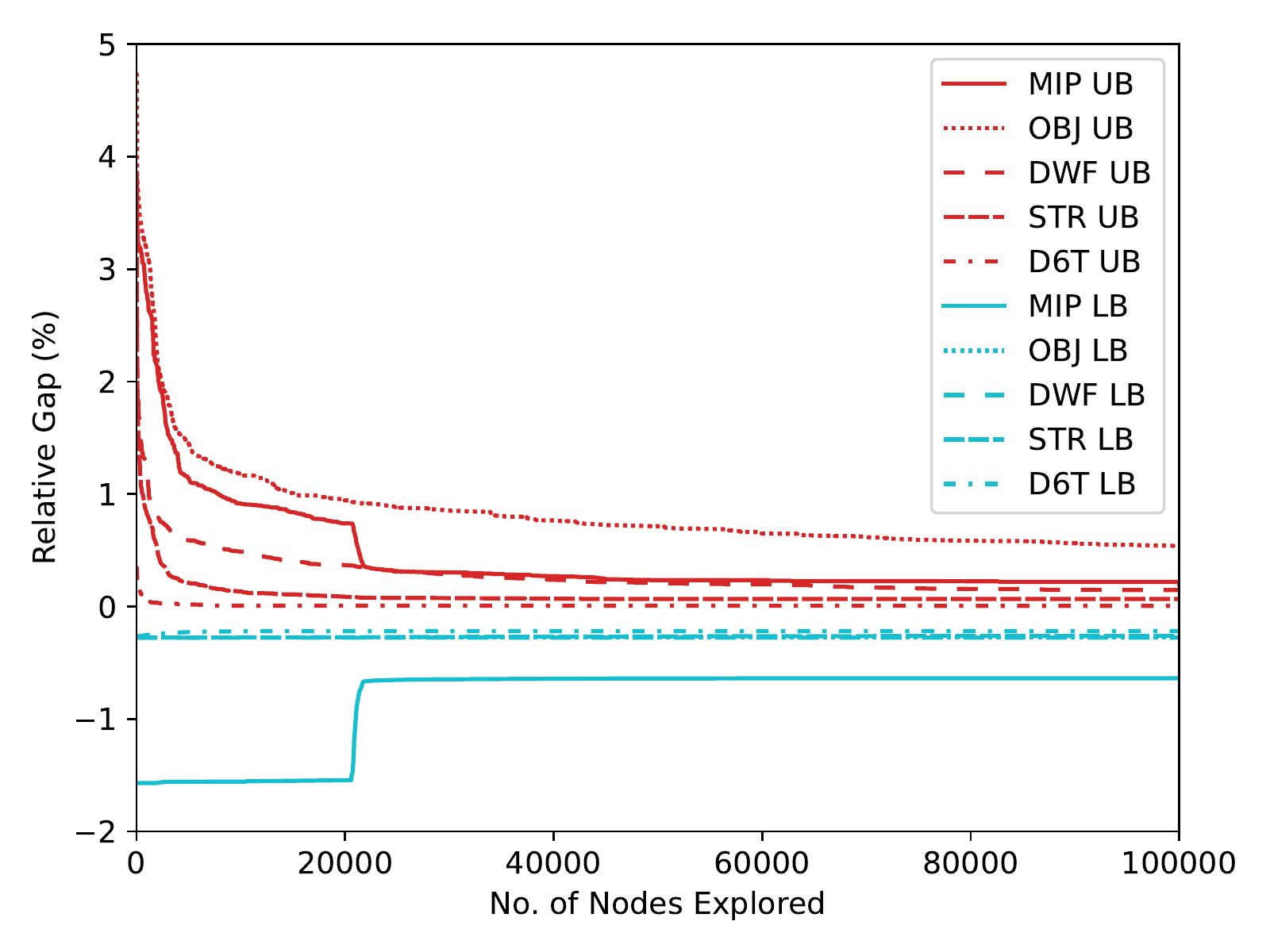}
\end{figure}

\subsection{MIP Experiments on TKP}
\label{sec:TKP}

Finally, in this section we report some experiments  conducted on hard TKP instances. In preliminary experiments, we observe that the tilting techniques introduced in Section \ref{subsec:tilting} are not very effective for TKP. This is mainly due to the fact that the DWB cuts after disjunctive coefficient strengthening are already defining high-dimensional faces (sometimes facets) of the block polyhedra for TKP. Therefore, we only present results regarding the performance of formulations {\MIP}, {\OBJ-$B$}, {\DWB-$B$} and {\STR-$B$} (defined at the beginning of Section \ref{sec:MKAP}) for $B\in\{32,64\}$, where $B$ indicates the number constraints used in each block to define the particular DW reformulation. Note that, unlike MKAP that has a loosely coupled block structure, the blocks in TKP's DW reformulations are overlapping, and in principle the DW relaxation associated with $B=64$ is stronger than the DW relaxation associated with $B=32$. It is also worth mentioning that we solve subproblems \eqref{Def:Dj}, \eqref{CT:subproblem0} and \eqref{CT:subproblem1} to exact optimality rather than the default 0.01\% optimality gap tolerance to avoid numerical issues.

\begin{table}[htb!]
	\centering
	\scriptsize
	\caption{Comparison of Gurobi Performance on TKP Instances}
	\label{Table:BnC_TKP}
	\begin{tabular}{crrrrrrrrrrrrrr}
		\toprule
	    Subclass & \multicolumn{7}{c}{Number of Instances Solved} & \multicolumn{7}{c}{Average Solution Time (s)}\\
		\cmidrule(lr){2-8}\cmidrule(lr){9-15}
		& {\MIP} & {\OBJ} & {\DWB} & {\STR}& {\OBJ} & {\DWB} & {\STR} & {\MIP} & {\OBJ} & {\DWB} & {\STR} & {\OBJ} & {\DWB} & {\STR}\\
        & & -32 & -32 & -32 & -64 & -64 & -64 & & -32 & -32 & -32 & -64 & -64 & -64\\
        \midrule\midrule\addlinespace
		XVIII & 5/10 & 0/10 & 2/10 & 5/10 & 0/10 & 1/10 & \textbf{8}/10 & $\geq360$ & $\geq600$ & $\geq520$ & $\geq381$ & $\geq600$ & $\geq567$ & $\geq264$\\
        XIX & \textbf{8}/10 & 0/10 & 4/10 & 7/10 & 0/10 & 5/10 & \textbf{8}/10 & $\geq232$ & $\geq600$ & $\geq412$ & $\geq266$ & $\geq600$ & $\geq457$ & $\geq160$\\
        XX & 6/10 & 0/10 & 2/10 & \textbf{8}/10 & 0/10 & 2/10 & 7/10 & $\geq296$ & $\geq600$ & $\geq516$ & $\geq287$ & $\geq600$ & $\geq557$ & $\geq212$\\
        \bottomrule
	\end{tabular}
\end{table}
\begin{table}[htb!]
	\centering
	\tiny
	\caption{Comparison of Gurobi Performance on TKP Instances (Cont'd)}
	\label{Table:BnC_TKP_condt}
	\begin{tabular}{crrrrrrrrrrrrrr}
		\toprule
		Subclass & \multicolumn{7}{c}{Average Optimality Gap (\%)} & \multicolumn{7}{c}{Average Number of Nodes}\\
		\cmidrule(lr){2-8}\cmidrule(lr){9-15}
		& {\MIP} & {\OBJ} & {\DWB} & {\STR}& {\OBJ} & {\DWB} & {\STR} & {\MIP} & {\OBJ-32} & {\DWB-32} & {\STR-32} & {\OBJ-64} & {\DWB-64} & {\STR-64}\\
        & & -32 & -32 & -32 & -64 & -64 & -64\\
        \midrule\midrule\addlinespace
        XVIII & 0.04 & 0.17 & 0.07 & 0.05 & 0.06 & 0.06 & \textbf{0.02} & $\geq$425495 & $\geq$2922669 & $\geq$1282107 & $\geq$\phantom{0}913574 & $\geq$2649095 & $\geq$1865134 & $\geq$949243\\
        XIX & \textbf{0.02} & 0.17 & 0.06 & 0.04 & 0.08 & 0.03 & \textbf{0.02} & $\geq$273020 & $\geq$2739441 & $\geq$\phantom{0}927412 & $\geq$1036383 & $\geq$2677712 & $\geq$1223015 & $\geq$574392\\
        XX & 0.04 & 0.17 & 0.08 & 0.04 & 0.10 & 0.07 & \textbf{0.03} & $\geq$315943 & $\geq$2449477 & $\geq$1054214 & $\geq$\phantom{0}849443 & $\geq$2469405 & $\geq$1839669 & $\geq$807173\\
		\bottomrule
	\end{tabular}
\end{table}


In Tables \ref{Table:BnC_TKP} and \ref{Table:BnC_TKP_condt}, we compare Gurobi performance of different formulations on TKP test instances, from the hard subclasses XVIII-XX in \cite{caprara2013uncommon}, with a 10-minute timelimit for running each formulation (excluding the time spent on generating the formulation). {We want to emphasize here that, unlike for MKAP, time spent on generating some formulations for TKP is not negligible.} Time spent on generating these different formulations is presented in Appendix \ref{apd:time_gen_TKP}. We observe that, although formulation {\DWB}-$B$ is not as effective as {\MIP} for $B\in\{32,64\}$, {\STR-32} is comparable with {\MIP} and {\STR-64} outperforms {\MIP} in terms of number of instances solved and average final optimality gap. We also note that the performance of formulations {\OBJ-$B$}, {\DWB-$B$} and {\STR-$B$} relatively increases as $B$ increases, illustrating the benefit of using DWB cuts generated from a tighter DW relaxation. Comparing {\DWB-$B$} and {\STR-$B$}, we observe a significant improvement in performance when coefficient strengthening is applied. This can be attributed to the fact that the TKP instances have relatively few blocks (32 blocks for $B=32$ and 16 blocks for $B=64$), in which case the last-iteration DWB cuts without strengthening would likely lead to a formulation with relatively high dual degeneracy.

\section{Hybrid Implementation in a Multi-Thread Context}
\label{sec:hybrid}
Computationally, using the last iteration DWB cuts (and their strengthened versions) does not always lead to a performance improvement as it $(i)$ requires additional computational effort to generate the cuts and $(ii)$ increases the size of the formulation. Similar to other MIP techniques, one needs to decide whether or not to use this approach for a given problem instance. Machine learning (ML) techniques have been widely investigated to make algorithmic decisions inside MIP solvers \cite{berthold2022learning,bonami2022classifier}.

In the remainder of this section, we present some preliminary experimental results with  a simple ML model to predict whether or not one may benefit from  DWB cuts on MKAP $I_1$ instances. We consider a computational environment with 4 threads where one starts with running the MIP solver on {\MIP} (using 3 threads) and solves the Lagrangian dual problem (using 1 thread) in parallel until the Lagrangian dual problem is solved. After this first phase, one decides to either allocate all 4 threads to the original formulation {\MIP} or to the strengthened formulation {\STR}.

To collect labels for the training phase, we use 20\% of MKAP $I_1$ instances as the training set. We first solve the original formulation (using 3 threads) and Lagrangian dual (using 1 thread) in parallel until the Lagrangian dual computation is completed, and collect some simple features for each instance. We then separately try allocating all 4 threads to (i) solving the original formulation {\MIP}, and (ii) solving the strengthened formulation {\STR}.
We label an instance as promising if one of the following holds when both methods are run for 10 minutes:
\begin{enumerate}
	\item The instance can be solved with {\STR} but not with {\MIP}, or,
	\item Neither method can solve the instance, and {\STR} has a smaller optimality gap (at least 0.01\% smaller), or,
	\item Both methods can solve the instance and {\STR} reduces the solution time by  10\% or more.
\end{enumerate}

Using this labeled data, we train a simple decision stump (depth-1 decision tree) to determine whether or not we should restart the MIP using the formulation {\STR}. The features we use are composites of $z_L$, $z_D$, $z_R$, $z_{\text{LB}}$ and $z_{\text{UB}}$, where $z_{\text{LB}}$ and $z_{\text{UB}}$ denote the lower and upper bounds obtained by the solver from solving the original formulation until the Lagrangian dual problem is solved. In our experiments, $z_{\text{UB}}$ is always finite as the solver can easily find out that the all-zero solution is feasible for MKAP.
The resulting  decision stump obtained after training is as follows:\begin{center}
	``If $(z_{D}-z_{\text{LB}})/z_{\text{UB}} >0.05\%$, then switch to the strengthened formulation {\STR}."
\end{center}

Instances that can be solved using the original MIP formulation before we finish computing $z_D$ are dropped from the training set and the test set as there is no algorithmic decision of interest to make in that situation. This leaves us 62 training instances and 242 test instances. 

Next, we compare a hybrid implementation with default Gurobi. To further simplify the implementation, we replace the parallelization of the solution of the original formulation and of the Lagrangian dual by a simulated parallelization. Such a simulated hybrid implementation is described below:
{\begin{itemize}
	\item[\bf Step 1.]\label{step:simulated1}
	\begin{itemize}
		\item[(a)] With 1 thread, solve problem \eqref{Lag_dual} with  level method. Store  solution time $t_D$.
		\item[(b)] With 3 threads, run a  MIP solver using  original MIP formulation with timelimit $t_D$. 
	\end{itemize}
	\item[\bf Step 2.] If the MIP is solved to optimality in Step 1, then stop.
	
	Else if $(z_{D}-z_{\text{LB}})/z_{\text{UB}} >0.05\%$, then restart the MIP solver with all 4 threads solving formulation {\STR} and using the best primal solution (bound) obtained in 1(b). 
	
	Else, allocate all $4$ threads to the MIP solver and continue with the original formulation.
\end{itemize}
}

\begin{table}[bt!]
	\centering
	\caption{Comparing {\MIP}, {\STR} and {\HYB} on MKAP Test Instances}
	\footnotesize
	\label{Table:multicore_compare}
	\begin{tabular}{lccc}
		\toprule
		& {\MIP} & {\STR} & {\HYB} \\
		\cmidrule{2-4}
		\text{Number of Instances Solved} & 167/242 & 206/242 & 208/242\\
		\text{Average Optimality Gap (\%)} & 0.18\% & 0.05\% & 0.05\%\\
		\text{Average Solution Time (s)} & $\geq 223$ & $\geq 117$ & $\geq 110$\\
		\bottomrule
	\end{tabular}
\end{table}

In Table \ref{Table:multicore_compare}  we summarize results obtained on MKAP by directly solving the original formulation (`{\MIP}'), directly generating and solving the strengthened formulation (`{\STR}'), and applying our simulated hybrid implementation (`{\HYB}'). For each method, we report the number of instances solved within the timelimit  of 10 minutes, the ending optimality gap  and average time spent on solving the problem including cut generating time for {\STR} and {\HYB} in seconds. Note that {\STR} is likely to outperform {\MIP} on the test instances because the DW bound is already significantly stronger than the LP relaxation bound on those instances. We observe that both {\STR} and {\HYB} outperform {\MIP} on the test instances while {\HYB} slightly outperforms {\MIP} in terms of solution time. This is because {\HYB} avoids solving the more expensive {\STR} formulation in some cases when the bound provided by {\STR} is not significantly stronger than {\MIP}. Another interesting observation is that {\HYB} solves exactly all instances that can be solved by either {\MIP} or {\STR}, including two instances that cannot be solved by {\STR} but can be solved by {\MIP}, demonstrating the value of having the flexibility of switching between different formulations. We believe this hybrid version would be relatively easy to implement by modern MIP solvers.

\section{Conclusions}
\label{sec:end}
In this paper, we develop practical methods to generate and strengthen cutting planes that one can derive from DW relaxation of MIPs with block structures. Numerical experiments show that adding these cutting planes is effective in cases when the DW bound is strong. We also describe how to incorporate our methodology into a MIP solver when multiple threads are available.

Empirically, we observe that adding too many cuts to the formulation can slow down LPs significantly. A potential fix is to consider cut selection within our cut generation approach. One may also consider adaptive tilting schemes rather than tilting each inequality to the same depth. We also observe that the Lagrangian dual problem often has nonunique optimal solutions. It is therefore interesting to investigate whether using alternative dual solutions can lead to stronger cuts.

\section*{Acknowledgements}
We are indebted to two anonymous referees whose detailed comments helped us to significantly improve the paper.

\bibliography{ref}
\bibliographystyle{ieeetr}

\newpage

%
%
%

\appendix
	\section{Dantzig-Wolfe Decomposition}\label{apd:algs}
	\begin{algorithm}[ht!]
		\scriptsize
		\begin{algorithmic}[1]
			\Initialize{$\hat{V}^j\gets$ a subset  of extreme points of $\conv(Q^j)$, $j=1,2,\ldots,q$\\
				$\hat{R}^j\gets$ a subset of extreme rays of $\conv(Q^j)$, $j=1,2,\ldots,q$\\ $t\gets 0$}
			\State $t\gets t+1$, solve\begin{equation}\label{DWmaster}
				\begin{alignedat}{3}
					{z}_D^t=\min\ & c^\top x\\
					\text{s.t. }&x_{I(j)}=\sum_{v\in \hat{V}^j}v\lambda^j_v+\sum_{r\in\hat{R}^j}r\mu_r^j,\quad &&j\in J,\quad&&(\pi^j)\\
					&Ax\geq b,\quad&&&&(\beta)\\
					&\sum_{v\in\hat{V}^j }\lambda^j_v=1,\quad &&j\in J,\quad&&(\theta_j)\\
					&\lambda^j\geq 0,~\mu^j\geq 0,\quad&& j\in J
				\end{alignedat}
			\end{equation}
			(assume $\{\hat{V}^j\}_{j=1}^{q}$ and $\{\hat{R}^j\}_{j=1}^{q}$ are initialized such that \eqref{DWmaster} is always feasible) \label{iter_start}
			\State let $(\pi^{1},\ldots,\pi^{q},\theta_{1},\ldots,\theta_{q})$ denote the  values of optimal dual variables for \eqref{DWmaster}\label{algline:def_pisig}
			\For{$j=1,2,\ldots,q$}
			\State solve the following pricing problem:\begin{equation}\label{DWpricingsub}
				D_j(\pi^{j}):=\min\{(\pi^{j})^\top v:v\in \conv(Q^j) \}=\min\{(\pi^{j})^\top v:v\in Q^j \}.
			\end{equation}
			\If{the pricing problem \eqref{DWpricingsub} is bounded}
			\State let $v^{j}$ denote an optimal solution
			\State $\zeta^{j}\gets(\pi^{j})^\top v^{j}-\theta_{j}$
			\If{$\zeta^{j}<0$}
			\State $\hat{V}^j\gets \hat{V}^j\cup\{v^{j}\}$\EndIf
			\Else 
			\State $\zeta^{j}\gets -\infty$
			\State let $r^{j}$ denote an extreme ray of $\conv(Q^j)$ with $(\pi^{j})^\top r^{j}<0$
			\State$\hat{R}^j\gets\hat{R}^j\cup\{r^{j}\}$
			\EndIf
			\If{$\zeta^{j}\geq 0$ for all $j\in J$}
			\State \Return $z_D={z}_D^t$
			\Else 
			\State \Goto{iter_start}
			\EndIf
			\EndFor 
		\end{algorithmic}
		\caption{Standard Dantzig-Wolfe Decomposition}
		\label{alg:DW}
	\end{algorithm}

	\section{The Level Method for the Dual Problem}\label{apd:level}
	
	\begin{algorithm}[ht!]
		\scriptsize
		\begin{algorithmic}[1]
			\Initialize{$\hat{V}^j\gets\emptyset$, $\hat{R}^j\gets\emptyset$, $j=1,2,\ldots,p$\\
				$\bar{z}\gets$ an upper bound of $z_D$\\
				LB $\gets -\infty$, UB $\gets \infty$, $t\gets 0$}
			\State \textbf{Main Loop:} $t\gets t+1$, \textbf{solve:}\begin{subequations}\label{Lag_master}
				\begin{alignat}{2}
					\text{UB}\gets\max\ &\sum_{j=1}^q \theta_j+b^\top \beta\\
					\text{s.t. } &\theta_j\leq v^\top \pi^j,&v\in\hat{V}^j,~j\in J,\label{dual_con1}\\
					&r^\top \pi^j\geq 0,&r\in\hat{R}^j,~j\in J,\label{dual_con2}\\
					&\sum_{j=1}^q \theta_j+b^\top\beta\leq \bar{z},\label{dual_con3}\\
					&\sum_{j:i\in I(j)}\pi^j_i+\beta^\top A_i=c_i, &i=1,\ldots,n,\label{dual_con4}\\
					&\beta\geq 0.\label{dual_con5}
				\end{alignat}
			\end{subequations}
			\label{level_iter_start}
			\If{$LB=-\infty$} 
			\State $(\bar{\pi},\bar{\beta})\gets$ optimal value of $(\pi,\beta)$ in \eqref{Lag_master}
			\Else
			\State \textbf{solve: }\begin{equation}\label{level_QP}
				\begin{alignedat}{3}
					\min\ &\big\|\big(\pi-\bar{\pi},\beta-\bar{\beta}\big)\big\|_2^2\\
					\text{s.t. }&\sum_{j=1}^q \theta_j+b^\top \beta\geq 0.7\cdot\text{UB}+0.3\cdot\text{LB}\\
					&\eqref{dual_con1}-\eqref{dual_con5}
				\end{alignedat}
			\end{equation}
			\State $(\bar{\pi},\bar{\beta})\gets$ optimal value of $(\pi,\beta)$ in \eqref{level_QP}
			\EndIf
			\For{$j=1,2,\ldots,q$}
			\State \textbf{solve} \eqref{Def:Dj} for $\pi^j=\bar{\pi}^j$
			\If{\eqref{Def:Dj} is bounded}
			\State let $v^j$ denote an optimal solution
			\State $\hat{V}^j\gets\hat{V}^j\cup\{v^j\}$
			\Else
			\State let $r^j$ denote an extreme ray of $\conv(Q^j)$ with $(\pi^j)^\top r^j<0$
			\State $\hat{R}^j\gets\hat{R}^j\cup\{r^j\}$
			\EndIf
			\EndFor
			\State LB $\gets\max\big\{\text{LB},\sum_{j=1}^q D_j(\bar{\pi}^j)+b^\top\bar{\beta}\big\}$
			\If{UB-LB is small enough}
			\State \Return LB
			\Else
			\State \Goto{level_iter_start}
			\EndIf
		\end{algorithmic}
		\caption{The Level Method for Solving \eqref{Lag_dual}}
		\label{alg:level}
	\end{algorithm}
	The level method adds on top of the cutting plane method a regularization step that requires solving a convex quadratic program to find a promising candidate multiplier $(\pi,\beta)=(\bar{\pi},\bar{\beta})$ to explore while staying close to the previous multiplier that is already explored. A pseudocode of (the multi-cut version of) the level method is given in Algorithm \ref{alg:level}. For generating the upper bound $\bar{z}$, we give the original formulation to the solver. We pick the second feasible solution $\bar{x}$ found by the solver (often much better than the first feasible solution) and set $\bar{z}=c^\top\bar{x}$. Constraint \eqref{dual_con3} is important in early iterations of the algorithm to ensure that problem \eqref{Lag_dual} is bounded, but becomes redundant when UB$<\bar{z}$ in later iterations of the algorithm. Within the level method, LB and UB store the best lower and upper bounds of $z_D$ found by the algorithm. If we set the termination condition to be UB-LB=0, then the algorithm returns the exact DW bound $z_D$. To avoid some numerical issues, when implementing the level method we terminate the algorithm if difference between LB and UB is within $10^{-6}$ relative tolerance, or if the solver fails to solve the quadratic master problem \eqref{level_QP}. We observe that the second case rarely happens but can sometimes lead to a Lagrangian dual bound weaker than the natural LP relaxation bound since we do not solve the Lagrangian dual problem to optimality. 
	
	\section{Examples}\label{apd:exmps}
	\begin{exmp}\label{exmp:intermediate_cuts}
		Consider the following MIP:\begin{align*}
			z^*=\min\ &x_1+x_2+2x_3+2x_4\\
			\text{s.t. } & x_2+x_4\geq 3,\\
			& 3x_1+x_2+3x_3+x_4\geq -12,\\
			& 0.5\leq x_i\leq 2.5,~x_i\in\bZ,~i=1,2,3,4.
		\end{align*}
		Define $I_1=\{1,2\}$, $I_2=\{3,4\}$, $Q^j=\{y\in\bZ^2:0.5\leq y_1\leq 2.5,0.5\leq y_2\leq 2.5\}$ for $j\in\{1,2\}$ and the linking constraints $Ax\geq b$ to be
		$$\left[\begin{array}{cc}			0  & 1 \\			3  & 1		\end{array}\right]
		\left[\begin{array}{cc}			x_1 \\			x_2		\end{array}		\right]
		+\left[\begin{array}{cc}			0  & 1 \\			3  & 1		\end{array}\right]
		\left[\begin{array}{cc}			x_3 \\			x_4		\end{array}		\right]
		\geq \left[\begin{array}{cc}			3 \\			12		\end{array}		\right].$$ 
		In this example, the MIP has symmetric blocks but asymmetric objective coefficients.
		The LP relaxation of the problem gives the lower bound $7$. Let $\pi(1)=(1,1/4,2,5/4)^\top $, $\beta(1)=(3/4,0)^\top $, $\pi(2)=(2/11,8/11,13/11,19/11)^\top $, $\beta(2)=(0,3/11)^\top $. Note that the dual multipliers $\big(\pi(\tau),\beta(\tau)\big)_{\tau\in\{1,2\}}$ satisfy the assumptions in Proposition \ref{prop:dual_bnd}. The best dual bound is given by\begin{displaymath}
			\max_{\tau\in\{1,2\}}z\big(\pi(\tau),\beta(\tau)\big)=\max\{27/4,78/11\}=78/11.
		\end{displaymath}
		Adding DWB cuts associated with $\pi(1)$ and adding DWB cuts associated with $\pi(2)$ into the LP relaxation of the original formulation yield bounds $125/16$ and $149/19$, respectively. These bounds are stronger than the corresponding Lagrangian relaxation bounds $27/4$ and $78/11$. On the other hand, the LP\begin{displaymath}
			\begin{aligned}
				\min\ &c^\top x\\
				\text{s.t. }&\big(\pi^j(\tau)\big)^\top x_{I(j)}\geq D_j\big(\pi^j(\tau)\big),\quad j\in J,~\tau=1,2,\\[.2cm]
				&Ax\geq b
			\end{aligned}
		\end{displaymath}
		has the optimal objective value equal to $8$, which happens to be $z^*$ in this example.
	\end{exmp}
	\begin{exmp}\label{exmp:unique_primal}
		Consider the following primal and dual LP pair:\begin{displaymath}
			\begin{alignedat}{2}
				(\text{Primal}):~\min~&&x_1-x_2\\
				\text{s.t.}~&&x_1+x_2&\geq 1,\\
				&&x_2&\geq 1,\\
				&&-x_2&\geq -1;
			\end{alignedat}
			\quad\begin{alignedat}{2}
				(\text{Dual}):~\max~&\myalpha_1+\myalpha_2-\myalpha_3\\
				\text{s.t.}~&\myalpha_1&&=1,\\
				&\myalpha_1+\myalpha_2-\myalpha_3&&=-1,\\
				&\myalpha\geq 0.
			\end{alignedat}
		\end{displaymath}
		In this example, the primal LP has a unique optimal solution $x=(0,1)^\top$. Therefore, the dimension of the optimal face is 0. Also, the problem has a unique dual \emph{basic} solution $\myalpha=(1,0,2)^\top$, whose 0-norm is 2, which is strictly less than $n$ minus the dimension of the optimal face. However, note that all dual solutions of the form $\myalpha=(1,\lambda,2+\lambda)^\top$ with $\lambda\geq 0$ are optimal for the dual LP, i.e., the problem has nonunique dual optimal solutions.
	\end{exmp}
	
	\section{MKAP and Test Instances} \label{apd:MKAP}
	In MKAP, there are a finite number of items, knapsacks and item classes. Each item belongs to exactly one item class. The decision maker has to pack a subset of items into knapsacks, subject to the constraints that only items belonging to the same item class with a total weight no larger than the capacity of the knapsack can be packed into the same knapsack, with the aim of maximizing the total profit of the packed items. Specifically, let $N$, $M$ and $K$ denote the index sets of items, knapsacks and item classes, respectively. For $j\in N$, let $p_j$ denote the profit of item $j$ and $w_j$ denote the weight of item $j$. For $i\in M$, let $C_i$ denote the capacity of knapsack $i$. For $k\in K$, let $S_k$ denote the set of items that belong to item class $k$, i.e., $\{S_k\}_{k\in K}$ is a partition of $N$. We use variables $x\in\{0,1\}^{M\times N}$ to represent the packing decisions, where $x_{ij}=1$ if and only if item $j$ is packed into knapsack $i$. Variables $y\in\{0,1\}^{M\times K}$ represent the assignment decisions, where $y_{ik}=1$ if and only if knapsack $i$ is used to pack items from item class $k$. MKAP can then be formulated as the following integer program (we flip the sign of the objective function to formulate it as a minimization problem):\begin{subequations}\label{MKAP}
		\begin{alignat}{2}
			\min~&-\sum_{i\in M}\sum_{j\in N}p_jx_{ij}\\
			\text{s.t. }&\sum_{j\in S_k}w_jx_{ij}\leq C_iy_{ik},&&i\in M,~k\in K,\label{MKAP:knapsack}\\
			&\sum_{i\in M}x_{ij}\leq 1,&&j\in N,\label{MKAP:assignment}\\
			&\sum_{k\in K}y_{ik}\leq 1,&&i\in M,\label{MKAP:class}\\
			&x\in\{0,1\}^{M\times N},~y\in\{0,1\}^{M\times K}.~~
		\end{alignat}
	\end{subequations}
	
	It has been observed in \cite{di2021new} that applying Lagrangian relaxation to \eqref{MKAP} with constraints \eqref{MKAP:assignment} and \eqref{MKAP:class} dualized can lead to a dual bound potentially much stronger than the LP relaxation bound. By equivalence between DW decomposition and Lagrangian relaxation, this observation also applies to DW decomposition. Specifically, we define $|M|\times|K|$ blocks in our DW decomposition, where for each $i\in M$ and $k\in K$ block $Q^{i,k}$ is defined by a knapsack constraint $\sum_{j\in S_k}w_jx_{ij}\leq C_iy_{ik}$ together with constraints forcing $(x_{ij})_{j\in S_k}$ and $y_{ik}$ to be binary.
	
	We generate the instances following the scheme of \cite{kataoka2014upper} but using different instance sizes. We consider instances with $|K|\in\{2,5,10,25\}$, $|M|\in\{10,20,30,40\}$ and $|N|\in\{50,100,200,300\}$. For each combination $(|K|,|M|,|N|)$ of the parameters, we generate three types (uncorrelated, weakly correlated, strongly correlated) of MKAP instances with 10 instances generated using 10 different random seeds for each type. We refer readers to \cite{kataoka2014upper} for a more detailed description of the instance generation procedure. It is worth mentioning that the item classes $\{S_k\}_{k\in K}$ all have equal sizes $|N|/|K|$ for our test instances.

     \section{TKP and Test Instances}\label{apd:TKP}
    TKP is defined by a set of items $N:=\{1,2,\ldots,n\}$. Each item $i\in N$ is associated with a profit $p_i$, a weight $w_i$, and a time period $[s_i,t_i)$ during which the item would be active. The decision maker decides to pack a subset of the items into the knapsack, so that at any time point, the total weight of the active items selected is at most $C$. We use variables $x\in\{0,1\}^N$ to represent the packing decision, where $x_i=1$ if and only if item $i\in N$ is packed into the knapsack. Let $S_j:=\{i:s_i\leq s_j,~t_i>s_j\}$ be the set of items that are active at time $s_j$ for $j\in N$. Then TKP can be formulated as the following integer program (we flip the sign of the objective function to formulate it as a minimization problem):
    \begin{subequations}\label{TKP}
        \begin{alignat}{2}
            \min~&-\sum_{i\in N} p_ix_i\\
            \text{s.t. }&\sum_{i\in S_j}w_ix_i\leq C,\quad&j\in N,\label{TKP:knapsack_constraints}\\
            &x\in\{0,1\}^N.
        \end{alignat}
    \end{subequations}

    \cite{caprara2013uncommon} suggest nonstandard DW reformulations of \eqref{TKP} by partitioning the constraints \eqref{TKP:knapsack_constraints} into different blocks, leading to the following  DW reformulation:
    \begin{subequations}\label{TKP:DW_reform}
        \begin{alignat}{2}
            \min~&-\sum_{i\in N} p_ix_i\\
            \text{s.t. }&x_{I(k)}\in \conv(Q^k),\quad&k\in K,\\
            &x\in\{0,1\}^N.
        \end{alignat}
    \end{subequations}
    where $Q^k$ is defined by $B$ consecutive constraints in the original problem and $I(k)$ contains the indices of the variables that appear in these constraints. Specifically, for $k\in K:=\{1,\ldots,\lceil n/B\rceil\}$, we have $I(k)=\bigcup_{j\in N_k}S_j$ where $N_k=\big\{(k-1)B+1,(k-1)B+2,\ldots,\min\{kB,n\}\big\}$, and $Q^k=\{x_{I(k)}\in\{0,1\}^{I(k)}:\sum_{i\in S_j}w_ix_i\leq C,j\in N_k\}$.
    
    As anticipated, according to empirical results in \cite{caprara2013uncommon}, the strength of the associated DW bound tends to increase as $B$ grows. In our numerical experiments, we consider two DW reformulations with group sizes $B=32$ and $B=64$, respectively. Since current commercial solvers (Gurobi in particular) can easily solve the original formulation for most test instances in \cite{caprara2013uncommon}, we consider only the 30 relatively hard instances from groups XVIII-XX in \cite{caprara2013uncommon}.

    \section{Dual Degeneracy of Different Formulations for MKAP}\label{apd:dual_degen}
    
    \begin{table}[tb!]
    	\scriptsize
    	\centering
    	\caption{Relative Degeneracy Levels of LP Optimal Dual Solutions for Different MKAP Formulations}
    	\label{Table:Comparison_OptFace}
    	\begin{tabular}{rrrrrrrrrrr}
    		\toprule
    		$(|K|,|M|,|N|)$ & \multicolumn{10}{c}{$(1-\|\myalpha^*\|_0/n)\times 100\%$}\\
    		\cmidrule(lr){2-11}
    		& \MIP & \OBJ & \DWB & \STR & \DkT{1} & \DkT{2} & \DkT{3} & \DkT{4} & \DkT{5} & \DkT{6} \\
    		\midrule\midrule\addlinespace
    		(\phantom{0}2,20,\phantom{0}50) & 56.27\% &	99.90\% &	78.30\% &	41.85\% &	36.37\% &	30.94\% &	25.77\% &	23.87\% &	21.86\% &	21.88\% \\
    		(\phantom{0}2,30,\phantom{0}50) & 57.75\% &	99.94\% &	70.21\% &	52.25\% &	46.32\% &	41.81\% &	39.24\% &	39.52\% &	39.89\% &	39.26\% \\
    		(\phantom{0}2,40,\phantom{0}50) & 58.06\% &	99.95\% &	71.04\% &	55.17\% &	51.67\% &	52.29\% &	51.82\% &	50.46\% &	50.46\% &	50.46\% \\
    		(\phantom{0}5,10,\phantom{0}50) & 50.84\% &	99.82\% &	77.44\% &	42.29\% &	32.33\% &	22.32\% &	14.84\% &	11.09\% &	11.25\% &	11.04\% \\
    		(\phantom{0}5,20,\phantom{0}50) & 53.20\% &	99.91\% &	68.12\% &	36.42\% &	28.41\% &	22.50\% &	20.10\% &	19.95\% &	19.04\% &	19.42\% \\
    		(\phantom{0}5,30,\phantom{0}50) & 54.60\% &	99.94\% &	65.54\% &	45.31\% &	34.81\% &	34.89\% &	35.45\% &	36.05\% &	36.11\% &	36.11\% \\
    		(\phantom{0}5,40,\phantom{0}50) & 54.90\% &	99.95\% &	70.72\% &	55.61\% &	51.21\% &	50.62\% &	50.61\% &	50.61\% &	50.61\% &	50.61\% \\
    		(10,10,\phantom{0}50) & 46.60\% &	99.83\% &	57.63\% &	26.99\% &	11.24\% &	\phantom{0}3.01\% &	\phantom{0}3.35\% &	\phantom{0}3.46\% &	\phantom{0}3.46\% &	\phantom{0}3.46\% \\
    		\textbf{(10,10,100)} & \textbf{50.65\%} &	\textbf{99.91\%} &	\textbf{56.98\%} &	\textbf{35.05\%} &	\textbf{22.84\%} &	\textbf{\phantom{0}7.77\%} &	\textbf{\phantom{0}1.36\%} &	\textbf{\phantom{0}1.14\%} &	\textbf{\phantom{0}1.49\%} &	\textbf{\phantom{0}1.54\%} \\
    		\textbf{(10,10,200)} & \textbf{53.30\%} &	\textbf{99.95\%} &	\textbf{53.77\%} &	\textbf{38.02\%} &	\textbf{26.06\%} &	\textbf{17.45\%} &	\textbf{\phantom{0}7.19\%} &	\textbf{\phantom{0}1.05\%} &	\textbf{\phantom{0}0.26\%} &	\textbf{\phantom{0}0.28\%} \\
    		\textbf{(10,10,300)} & \textbf{54.23\%} &	\textbf{99.97\%} &	\textbf{53.50\%} &	\textbf{39.77\%} &	\textbf{26.75\%} &	\textbf{20.75\%} &	\textbf{17.99\%} &	\textbf{\phantom{0}7.47\%} &	\textbf{\phantom{0}1.43\%} &	\textbf{\phantom{0}0.61\%} \\
    		(10,20,\phantom{0}50) & 48.77\% &	99.92\% &	57.55\% &	24.16\% &	16.96\% &	11.23\% &	11.46\% &	11.36\% &	11.36\% &	11.36\% \\
    		(10,30,\phantom{0}50) & 50.05\% &	99.94\% &	57.02\% &	38.97\% &	33.06\% &	31.63\% &	31.42\% &	31.42\% &	31.42\% &	31.42\% \\
    		(10,40,\phantom{0}50) & 50.32\% &	99.96\% &	62.98\% &	45.03\% &	44.05\% &	43.36\% &	43.36\% &	43.36\% &	43.36\% &	43.36\% \\
    		(10,40,100) & 54.98\% &	99.98\% &	80.55\% &	45.61\% &	35.69\% &	27.61\% &	25.77\% &	24.92\% &	25.41\% &	25.50\% \\
    		(25,10,\phantom{0}50) & 37.28\% &	99.87\% &	26.68\% &	11.18\% &	\phantom{0}5.86\% &	\phantom{0}5.86\% &	\phantom{0}5.86\% &	\phantom{0}5.86\% &	\phantom{0}5.86\% &	\phantom{0}5.86\% \\
    		\textbf{(25,10,100)} & \textbf{44.57\%} &	\textbf{99.92\%} &	\textbf{40.47\%} &	\textbf{30.85\%} &	\textbf{14.51\%} &	\textbf{\phantom{0}4.40\%} &	\textbf{\phantom{0}3.70\%} &	\textbf{\phantom{0}3.70\%} &	\textbf{\phantom{0}3.70\%} &	\textbf{\phantom{0}3.70\%} \\
    		\textbf{(25,10,200)} & \textbf{49.75\%} &	\textbf{99.96\%} &	\textbf{47.23\%} &	\textbf{37.99\%} &	\textbf{29.72\%} &	\textbf{13.19\%} &	\textbf{\phantom{0}3.71\%} &	\textbf{\phantom{0}1.59\%} &	\textbf{\phantom{0}1.36\%} &	\textbf{\phantom{0}1.68\%} \\
    		\textbf{(25,10,300)} & \textbf{51.73\%} &	\textbf{99.97\%} &	\textbf{56.84\%} &	\textbf{48.89\%} &	\textbf{34.16\%} &	\textbf{24.02\%} &	\textbf{\phantom{0}8.05\%} &	\textbf{\phantom{0}1.77\%} &	\textbf{\phantom{0}1.03\%} &	\textbf{\phantom{0}1.17\%} \\
    		(25,20,\phantom{0}50) & 39.01\% &	99.93\% &	34.14\% &	11.26\% &	\phantom{0}6.44\% &	\phantom{0}6.44\% &	\phantom{0}6.44\% &	\phantom{0}6.44\% &	\phantom{0}6.44\% &	\phantom{0}6.44\% \\
    		(25,20,100) & 47.02\% &	99.96\% &	52.18\% &	26.21\% &	10.38\% &	\phantom{0}3.34\% &	\phantom{0}3.57\% &	\phantom{0}3.57\% &	\phantom{0}3.57\% &	\phantom{0}3.57\% \\
    		\textbf{(25,20,200)} & \textbf{52.44\%} &	\textbf{99.98\%} &	\textbf{62.38\%} &	\textbf{38.30\%} &	\textbf{24.36\%} &	\textbf{\phantom{0}7.41\%} &	\textbf{\phantom{0}0.93\%} &	\textbf{\phantom{0}0.91\%} &	\textbf{\phantom{0}1.06\%} &	\textbf{\phantom{0}1.25\%} \\
    		\textbf{(25,20,300)} & \textbf{54.69\%} &	\textbf{99.98\%} &	\textbf{62.72\%} &	\textbf{42.19\%} &	\textbf{30.72\%} &	\textbf{22.26\%} &	\textbf{\phantom{0}4.78\%} &	\textbf{\phantom{0}0.56\%} &	\textbf{\phantom{0}0.53\%} &	\textbf{\phantom{0}0.62\%} \\
    		(25,30,\phantom{0}50) & 40.04\% &	99.96\% &	41.02\% &	23.72\% &	22.94\% &	22.94\% &	22.94\% &	22.94\% &	22.94\% &	22.94\% \\
    		(25,30,100) & 47.84\% &	99.97\% &	64.22\% &	32.04\% &	16.22\% &	\phantom{0}7.41\% &	\phantom{0}7.57\% &	\phantom{0}7.57\% &	\phantom{0}7.57\% &	\phantom{0}7.57\% \\
    		(25,30,200) & 53.39\% &	99.99\% &	73.83\% &	43.35\% &	30.70\% &	15.59\% &	\phantom{0}6.42\% &	\phantom{0}5.51\% &	\phantom{0}5.67\% &	\phantom{0}5.89\% \\
    		\textbf{(25,30,300)} & \textbf{55.68\%} &	\textbf{99.99\%} &	\textbf{78.99\%} &	\textbf{48.56\%} &	\textbf{41.44\%} &	\textbf{26.72\%} &	\textbf{\phantom{0}8.59\%} &	\textbf{\phantom{0}4.23\%} &	\textbf{\phantom{0}3.58\%} &	\textbf{\phantom{0}3.27\%} \\
    		(25,40,\phantom{0}50) & 40.26\% &	99.97\% &	50.78\% &	39.69\% &	39.52\% &	39.52\% &	39.52\% &	39.52\% &	39.52\% &	39.52\% \\
    		(25,40,100) & 48.39\% &	99.98\% &	61.63\% &	29.76\% &	20.09\% &	17.55\% &	17.17\% &	17.17\% &	17.17\% &	17.17\% \\
    		\bottomrule
    	\end{tabular}
    \end{table}

    In Table \ref{Table:Comparison_OptFace}, we report the relative normalized degeneracy level of the optimal dual solution computed by Gurobi for LP relaxations of different formulations for MKAP. The relative degeneracy level is calculated as the degeneracy level $n-\|\myalpha^*\|_0$ divided by the dimension $n$, i.e., $1-\|\myalpha^*\|_0/n$. Bold rows correspond to $I_2$ instances. 
    Due to numerical tolerances, the optimal values $\myalpha^*_k$ of some dual variables can sometimes be very close to but not equal to zero, especially the ones associated with bound constraints. We treat $\myalpha^*_k$ as 0 if $|\myalpha^*_k|\leq 10^{-8}$ for all entries $\myalpha^*_k$ of $\myalpha^*$ when computing $\|\myalpha^*\|_0$.
    Although the strengthening techniques in Section \ref{sec:Strengthening} do not guarantee a monotonic decrease in degeneracy level, it is clear from Table \ref{Table:Comparison_OptFace} that formulations with stronger strengthening tend to have less degenerate optimal dual solutions. We observe that the degeneracy level of {\DWB} is already comparable to {\MIP} on $I_2$ instances but is higher than {\MIP} on $I_1\setminus I_2$ instances. With coefficient strengthening, {\STR} has lower average dual degeneracy than {\MIP} in almost all cases except on the instance class (5,40,50). The dual degeneracy continues to decrease with further tilting. Finally, it is worth noting that, as expected, the relative degeneracy level of {\OBJ} is extremely high, close to 100\%.
	
	\section{Time Spent on Obtaining Different MKAP Formulations}\label{apd:time_gen}
	\begin{table}[hbt!]
		\scriptsize
		\centering
		\caption{Comparison of Time Spent on Obtaining Different Formulations for MKAP}
		\label{Table:Comparison_GenTime}
		\begin{tabular}{crrrrrrr}
			\toprule
			$(|K|,|M|,|N|)$ & \multicolumn{7}{c}{Generation Time Excluding Lagr. Dual Time (s)}\\
			\cmidrule(lr){2-8}
			& \STR & \DkT{1} & \DkT{2} & \DkT{3} & \DkT{4} & \DkT{5} & \DkT{6} \\
			\midrule\midrule\addlinespace
			(\phantom{0}2,20,\phantom{0}50) & 0.1 & 0.3 & 0.5 & 1.1 & 2.1 & 3.7 & 6.0\\
			(\phantom{0}2,30,\phantom{0}50) & 0.1 & 0.2 & 0.3 & 0.6 & 0.9 & 1.0 & 1.2\\
			(\phantom{0}2,40,\phantom{0}50) & 0.1 & 0.2 & 0.1 & 0.2 & 0.2 & 0.2 & 0.2\\
			(\phantom{0}5,10,\phantom{0}50) & 0.0 & 0.1 & 0.4 & 0.8 & 1.6 & 2.6 & 3.6\\
			(\phantom{0}5,20,\phantom{0}50) & 0.1 & 0.2 & 0.5 & 0.9 & 1.3 & 1.5 & 1.6\\
			(\phantom{0}5,30,\phantom{0}50) & 0.1 & 0.2 & 0.3 & 0.5 & 0.5 & 0.5 & 0.5\\
			(\phantom{0}5,40,\phantom{0}50) & 0.1 & 0.1 & 0.1 & 0.1 & 0.1 & 0.1 & 0.1\\
			(10,10,\phantom{0}50) & 0.0 & 0.2 & 0.4 & 0.8 & 1.1 & 1.2 & 1.2\\
			\textbf{(10,10,100)} & \textbf{0.1} & \textbf{0.3} & \textbf{1.0} & \textbf{2.2} & \textbf{4.5} & \textbf{8.7} & \textbf{15.7}\\
			\textbf{(10,10,200)} & \textbf{0.3} & \textbf{0.9} & \textbf{2.2} & \textbf{5.0} & \textbf{10.3} & \textbf{20.9} & \textbf{42.4}\\
			\textbf{(10,10,300)} & \textbf{0.9} & \textbf{2.9} & \textbf{7.0} & \textbf{15.2} & \textbf{30.5} & \textbf{60.0} & \textbf{118.7}\\
			(10,20,\phantom{0}50) & 0.1 & 0.2 & 0.4 & 0.6 & 0.7 & 0.7 & 0.7\\
			(10,30,\phantom{0}50) & 0.1 & 0.2 & 0.3 & 0.3 & 0.3 & 0.3 & 0.3\\
			(10,40,\phantom{0}50) & 0.1 & 0.1 & 0.1 & 0.1 & 0.1 & 0.1 & 0.1\\
			(10,40,100) & 0.3 & 0.9 & 2.0 & 3.8 & 5.5 & 6.6 & 6.9\\
			(25,10,\phantom{0}50) & 0.0 & 0.2 & 0.3 & 0.3 & 0.3 & 0.3 & 0.3\\
			\textbf{(25,10,100)} & \textbf{0.0} & \textbf{0.4} & \textbf{1.1} & \textbf{2.1} & \textbf{2.8} & \textbf{2.8} & \textbf{2.8}\\
			\textbf{(25,10,200)} & \textbf{0.1} & \textbf{0.7} & \textbf{2.1} & \textbf{4.9} & \textbf{10.3} & \textbf{20.2} & \textbf{38.7}\\
			\textbf{(25,10,300)} & \textbf{0.1} & \textbf{0.9} & \textbf{3.0} & \textbf{7.2} & \textbf{15.5} & \textbf{31.7} & \textbf{63.5}\\
			(25,20,\phantom{0}50) & 0.1 & 0.2 & 0.2 & 0.2 & 0.2 & 0.2 & 0.2\\
			(25,20,100) & 0.1 & 0.6 & 1.6 & 2.6 & 3.1 & 3.1 & 3.1\\
			\textbf{(25,20,200)} & \textbf{0.2} & \textbf{1.3} & \textbf{4.0} & \textbf{9.0} & \textbf{17.7} & \textbf{32.1} & \textbf{53.9}\\
			\textbf{(25,20,300)} & \textbf{0.4} & \textbf{2.0} & \textbf{5.8} & \textbf{13.5} & \textbf{28.1} & \textbf{56.1} & \textbf{109.5}\\
			(25,30,\phantom{0}50) & 0.1 & 0.2 & 0.2 & 0.2 & 0.2 & 0.2 & 0.2\\
			(25,30,100) & 0.2 & 0.8 & 1.7 & 2.4 & 2.5 & 2.5 & 2.5\\
			(25,30,200) & 0.4 & 1.9 & 5.4 & 11.8 & 21.9 & 36.1 & 51.1\\
			\textbf{(25,30,300)} & \textbf{0.8} & \textbf{3.1} & \textbf{8.3} & \textbf{19.3} & \textbf{39.7} & \textbf{77.9} & \textbf{148.1}\\
			(25,40,\phantom{0}50) & 0.1 & 0.1 & 0.0 & 0.0 & 0.0 & 0.0 & 0.0\\
			(25,40,100) & 0.2 & 0.8 & 1.5 & 1.9 & 2.0 & 2.0 & 2.0\\
			\bottomrule
		\end{tabular}
	\end{table}
    Since both {\OBJ} and {\DWB} can be obtained directly after solving the Lagrangian dual problem, we only report in Table \ref{Table:Comparison_GenTime} the total time spent on applying strengthening and different levels of tilting for {\STR} and \DkT{d} {} with $d\in\{1,\ldots,6\}$. Bold rows correspond to $I_2$ instances and other rows correspond to $I_1\setminus I_2$ MKAP instances. We note that, given a dual optimal solution, it is extremely efficient to generate the {\STR} formulation for MKAP while the time spent on generating \DkT{d} {} grows exponentially as $d$ increases.
	\section{Number of Cutting Planes Generated Using Different Formulations}\label{apd:nCuts}
	In Table \ref{Table:nCuts}, we report the number of cutting planes generated by Gurobi within 10 minutes for formulations {\MIP}, {\OBJ}, {\DWB}, {\STR}, {\DkT{3} } and {\DkT{6} } averaged over the (25, 20, 300) MKAP instances.
	\begin{table}[hbt!]
		\centering
		\caption{Average Number of Cutting Planes Generated by Gurobi}
		\label{Table:nCuts}
		\scriptsize
		\begin{tabular}{crrrrrr}
			\toprule
			& {\MIP} & {\OBJ} & {\DWB} & {\STR} & {\DkT{3} } & {\DkT{6} } \\
			\cmidrule{2-7}
			Gomory & 132.1 & 41.5 & 10.4 & 17.9 & 0.4 & 0.0\\
			Lift-and-project & 5.3 & 1.7 & 0.6 & 0.6 & 0.0 & 0.0\\
			Cover & 616.1 & 176.5 & 121.5 & 173.3 & 16.3 & 0.0\\
			Clique & 362.3 & 248.5 & 148.3 & 266.4 & 108.3 & 0.0\\
			MIR & 90.2 & 24.6 & 9.7 & 13.2 & 3.1 & 0.0\\
			StrongCG & 94.3 & 49.8 & 12.7 & 21.9 & 2.4 & 0.0\\
			Flow cover & 228.8 & 61.6 & 18.5 & 3.1 & 0.0 & 0.0\\
			Zero half & 34.9 & 13.9 & 2.0 & 0.4 & 0.0 & 0.0\\
			RLT & 19.7 & 0.8 & 1.4 & 0.7 & 0.0 & 0.0\\
			Relax-and-lift & 10.0 & 0.0 & 0.1 & 0.0 & 0.0 & 0.0\\
			\bottomrule
		\end{tabular}
	\end{table}
	
	\section{Results for $I_1\setminus I_2$ MKAP instances}\label{apd:bnc_I2notI1}
	\begin{table}[hbt!]
		\centering
		\scriptsize
		\caption{Comparison of Gurobi Performance on $I_1\setminus I_2$ MKAP Instances}
		\label{Table:BnC_I1}
		\begin{tabular}{crrrrrrrrrrrr}
			\toprule
			$(|K|,|M|,|N|)$ & \multicolumn{6}{c}{Number of Solved} & \multicolumn{6}{c}{Average Solution Time (s)}\\
			\cmidrule(lr){2-7}\cmidrule(lr){8-13}
			& {\MIP} & {\OBJ} & {\DWB} & {\STR} & {\DkT{3} } & {\DkT{6} } & {\MIP} & {\OBJ} & {\DWB} & {\STR} & {\DkT{3} } & {\DkT{6} }\\
			\midrule\midrule\addlinespace
			(\phantom{0}2,20,\phantom{0}50) & \textbf{30}/30 & 23/30 & \textbf{30}/30 & \textbf{30}/30 & \textbf{30}/30 & \textbf{30}/30 & \textbf{2} & $\geq$202 & 7 & 5 & 4 & 4\\
			(\phantom{0}2,30,\phantom{0}50) & \textbf{30}/30 & 29/30 & \textbf{30}/30 & \textbf{30}/30 & \textbf{30}/30 & \textbf{30}/30 & \textbf{0} & $\geq$\phantom{0}25 & \textbf{0} & \textbf{0} & \textbf{0} & \textbf{0}\\
			(\phantom{0}2,40,\phantom{0}50) & \textbf{30}/30 & \textbf{30}/30 & \textbf{30}/30 & \textbf{30}/30 & \textbf{30}/30 & \textbf{30}/30 & \textbf{0} & \textbf{0} & \textbf{0} & \textbf{0} & \textbf{0} & \textbf{0}\\
			(\phantom{0}5,10,\phantom{0}50) & \textbf{30}/30 & \textbf{30}/30 & \textbf{30}/30 & \textbf{30}/30 & \textbf{30}/30 & \textbf{30}/30 & 4 & 62 & 7 & 4 & 2 & \textbf{1}\\
			(\phantom{0}5,20,\phantom{0}50) & \textbf{30}/30 & \textbf{30}/30 & \textbf{30}/30 & \textbf{30}/30 & \textbf{30}/30 & \textbf{30}/30 & 1 & 42 & 4 & 2 & 1 & \textbf{0}\\
			(\phantom{0}5,30,\phantom{0}50) & \textbf{30}/30 & \textbf{30}/30 & \textbf{30}/30 & \textbf{30}/30 & \textbf{30}/30 & \textbf{30}/30 & \textbf{0} & 19 & \textbf{0} & \textbf{0} & \textbf{0} & \textbf{0}\\
			(\phantom{0}5,40,\phantom{0}50) & \textbf{30}/30 & \textbf{30}/30 & \textbf{30}/30 & \textbf{30}/30 & \textbf{30}/30 & \textbf{30}/30 & \textbf{0} & \textbf{0} & \textbf{0} & \textbf{0} & \textbf{0} & \textbf{0}\\
			(10,10,\phantom{0}50) & \textbf{30}/30 & \textbf{30}/30 & \textbf{30}/30 & \textbf{30}/30 & \textbf{30}/30 & \textbf{30}/30 & \textbf{0} & 1 & \textbf{0} & \textbf{0} & \textbf{0} & \textbf{0}\\
			(10,20,\phantom{0}50) & \textbf{30}/30 & \textbf{30}/30 & \textbf{30}/30 & \textbf{30}/30 & \textbf{30}/30 & \textbf{30}/30 & \textbf{0} & 16 & \textbf{0} & \textbf{0} & \textbf{0} & \textbf{0}\\
			(10,30,\phantom{0}50) & \textbf{30}/30 & \textbf{30}/30 & \textbf{30}/30 & \textbf{30}/30 & \textbf{30}/30 & \textbf{30}/30 & \textbf{0} & 1 & \textbf{0} & \textbf{0} & \textbf{0} & \textbf{0}\\
			(10,40,\phantom{0}50) & \textbf{30}/30 & \textbf{30}/30 & \textbf{30}/30 & \textbf{30}/30 & \textbf{30}/30 & \textbf{30}/30 & \textbf{0} & \textbf{0} & \textbf{0} & \textbf{0} & \textbf{0} & \textbf{0}\\
			(10,40,100) & \textbf{30}/30 & 1/30 & 29/30 & \textbf{30}/30 & \textbf{30}/30 & \textbf{30}/30 & 40 & $\geq$600 & $\geq$103 & 72 & 15 & \textbf{12}\\
			(25,10,\phantom{0}50) & \textbf{30}/30 & \textbf{30}/30 & \textbf{30}/30 & \textbf{30}/30 & \textbf{30}/30 & \textbf{30}/30 & \textbf{0} & \textbf{0} & \textbf{0} & \textbf{0} & \textbf{0} & \textbf{0}\\
			(25,20,\phantom{0}50) & \textbf{30}/30 & \textbf{30}/30 & \textbf{30}/30 & \textbf{30}/30 & \textbf{30}/30 & \textbf{30}/30 & \textbf{0} & \textbf{0} & \textbf{0} & \textbf{0} & \textbf{0} & \textbf{0}\\
			(25,20,100) & \textbf{30}/30 & 20/30 & \textbf{30}/30 & \textbf{30}/30 & \textbf{30}/30 & \textbf{30}/30 & 1 & $\geq$204 & \textbf{0} & \textbf{0} & \textbf{0} & \textbf{0}\\
			(25,30,\phantom{0}50) & \textbf{30}/30 & \textbf{30}/30 & \textbf{30}/30 & \textbf{30}/30 & \textbf{30}/30 & \textbf{30}/30 & \textbf{0} & \textbf{0} & \textbf{0} & \textbf{0} & \textbf{0} & \textbf{0}\\
			(25,30,100) & \textbf{30}/30 & 10/30 & \textbf{30}/30 & \textbf{30}/30 & \textbf{30}/30 & \textbf{30}/30 & 5 & $\geq$432 & 1 & 1 & \textbf{0} & \textbf{0}\\
			(25,30,200) & 16/30 & 1/30 & 10/30 & 15/30 & 23/30 & \textbf{26}/30 & $\geq$377 & $\geq$600 & $\geq$475 & $\geq$402 & $\geq$255 & $\geq$143\\
			(25,40,\phantom{0}50) & \textbf{30}/30 & \textbf{30}/30 & \textbf{30}/30 & \textbf{30}/30 & \textbf{30}/30 & \textbf{30}/30 & \textbf{0} & \textbf{0} & \textbf{0} & \textbf{0} & \textbf{0} & \textbf{0}\\
			(25,40,100) & \textbf{30}/30 & 9/30 & \textbf{30}/30 & \textbf{30}/30 & \textbf{30}/30 & \textbf{30}/30 & 2 & $\geq$464 & 1 & 1 & \textbf{0} & \textbf{0}\\
			\bottomrule
			\multicolumn{7}{l}{At each row, the averages of 30 instances are reported.}
		\end{tabular}
	\end{table}
	\begin{table}[hbt!]
		\centering
		\scriptsize
		\caption{Comparison of Gurobi Performance on $I_1\setminus I_2$ MKAP Instances (Cont'd)}
		\label{Table:BnC_I1_condt}
		\begin{tabular}{crrrrrrrrrrrr}
			\toprule
			$(|K|,|M|,|N|)$ & \multicolumn{6}{c}{Average Optimality Gap (\%)} & \multicolumn{6}{c}{Average Number of Nodes}\\
			\cmidrule(lr){2-7}\cmidrule(lr){8-13}
			& {\MIP} & {\OBJ} & {\DWB} & {\STR} & {\DkT{3} } & {\DkT{6} } & {\MIP} & {\OBJ} & {\DWB} & {\STR} & {\DkT{3} } & {\DkT{6} }\\
			\midrule\midrule\addlinespace
			(\phantom{0}2,20,\phantom{0}50) & \textbf{0.00} & 0.06 & \textbf{0.00} & \textbf{0.00} & \textbf{0.00} & \textbf{0.00} & \textbf{3439} & $\geq$755510 & 28880 & 16808 & 16826 & 7963\\
			(\phantom{0}2,30,\phantom{0}50) & \textbf{0.00} & \textbf{0.00} & \textbf{0.00} & \textbf{0.00} & \textbf{0.00} & \textbf{0.00} & \textbf{89} & $\geq$177782 & 341 & 796 & 188 & 131\\
			(\phantom{0}2,40,\phantom{0}50) & \textbf{0.00} & \textbf{0.00} & \textbf{0.00} & \textbf{0.00} & \textbf{0.00} & \textbf{0.00} & \textbf{1} & \textbf{1} & \textbf{1} & \textbf{1} & \textbf{1} & \textbf{1}\\
			(\phantom{0}5,10,\phantom{0}50) & \textbf{0.00} & \textbf{0.00} & \textbf{0.00} & \textbf{0.00} & \textbf{0.00} & \textbf{0.00} & 12677 & 669382 & 62160 & 23721 & 5498 & \textbf{1509}\\
			(\phantom{0}5,20,\phantom{0}50) & \textbf{0.00} & \textbf{0.00} & \textbf{0.00} & \textbf{0.00} & \textbf{0.00} & \textbf{0.00} & 1865 & 334795 & 20703 & 9992 & 929 & \textbf{614}\\
			(\phantom{0}5,30,\phantom{0}50) & \textbf{0.00} & \textbf{0.00} & \textbf{0.00} & \textbf{0.00} & \textbf{0.00} & \textbf{0.00} & 275 & 151202 & 161 & 61 & \textbf{17} & 18\\
			(\phantom{0}5,40,\phantom{0}50) & \textbf{0.00} & \textbf{0.00} & \textbf{0.00} & \textbf{0.00} & \textbf{0.00} & \textbf{0.00} & \textbf{1} & \textbf{1} & \textbf{1} & \textbf{1} & \textbf{1} & \textbf{1}\\
			(10,10,\phantom{0}50) & \textbf{0.00} & \textbf{0.00} & \textbf{0.00} & \textbf{0.00} & \textbf{0.00} & \textbf{0.00} & 1077 & 3227 & 39 & 51 & \textbf{1} & \textbf{1}\\
			(10,20,\phantom{0}50) & \textbf{0.00} & \textbf{0.00} & \textbf{0.00} & \textbf{0.00} & \textbf{0.00} & \textbf{0.00} & 344 & 46614 & 803 & 379 & \textbf{14} & \textbf{14}\\
			(10,30,\phantom{0}50) & \textbf{0.00} & \textbf{0.00} & \textbf{0.00} & \textbf{0.00} & \textbf{0.00} & \textbf{0.00} & 131 & 3045 & 372 & 129 & \textbf{1} & \textbf{1}\\
			(10,40,\phantom{0}50) & \textbf{0.00} & \textbf{0.00} & \textbf{0.00} & \textbf{0.00} & \textbf{0.00} & \textbf{0.00} & \textbf{1} & \textbf{1} & \textbf{1} & \textbf{1} & \textbf{1} & \textbf{1}\\
			(10,40,100) & \textbf{0.00} & 0.14 & 0.01 & 0.01 & \textbf{0.00} & \textbf{0.00} & 23261 & $\geq$302450 & $\geq$137206 & 82878 & 11574 & \textbf{9120}\\
			(25,10,\phantom{0}50) & \textbf{0.00} & \textbf{0.00} & \textbf{0.00} & \textbf{0.00} & \textbf{0.00} & \textbf{0.00} & 2 & 229 & \textbf{1} & \textbf{1} & \textbf{1} & \textbf{1}\\
			(25,20,\phantom{0}50) & \textbf{0.00} & \textbf{0.00} & \textbf{0.00} & \textbf{0.00} & \textbf{0.00} & \textbf{0.00} & \textbf{1} & 212 & \textbf{1} & \textbf{1} & \textbf{1} & \textbf{1}\\
			(25,20,100) & \textbf{0.00} & 0.03 & \textbf{0.00} & \textbf{0.00} & \textbf{0.00} & \textbf{0.00} & 279 & $\geq$202331 & \textbf{1} & \textbf{1} & \textbf{1} & \textbf{1}\\
			(25,30,\phantom{0}50) & \textbf{0.00} & \textbf{0.00} & \textbf{0.00} & \textbf{0.00} & \textbf{0.00} & \textbf{0.00} & \textbf{1} & 2 & \textbf{1} & \textbf{1} & \textbf{1} & \textbf{1}\\
			(25,30,100) & \textbf{0.00} & 0.15 & \textbf{0.00} & \textbf{0.00} & \textbf{0.00} & \textbf{0.00} & 9866 & $\geq$185765 & 1903 & 1121 & \textbf{1} & \textbf{1}\\
			(25,30,200) & 0.12 & 0.44 & 0.22 & 0.15 & 0.07 & \textbf{0.04} & $\geq$45681 & $\geq$341784 & $\geq$\phantom{0}67369 & $\geq$48540 & $\geq$20550 & $\geq$14553\\
			(25,40,\phantom{0}50) & \textbf{0.00} & \textbf{0.00} & \textbf{0.00} & \textbf{0.00} & \textbf{0.00} & \textbf{0.00} & \textbf{1} & 28 & \textbf{1} & \textbf{1} & \textbf{1} & \textbf{1}\\
			(25,40,100) & \textbf{0.00} & 0.08 & \textbf{0.00} & \textbf{0.00} & \textbf{0.00} & \textbf{0.00} & 450 & $\geq$257189 & 1520 & 2483 & \textbf{47} & \textbf{47}\\
			\bottomrule
			\multicolumn{7}{l}{At each row, the averages of 30 instances are reported.}
		\end{tabular}
	\end{table}
    In Tables \ref{Table:BnC_I1} and \ref{Table:BnC_I1_condt}, we report computational results for $I_1\setminus I_2$ MKAP instances. We observe that the original {\MIP} formulation is very competitive on these instances, whose performance is often better than {\DWB} and even {\STR}. This can be explained by the fact that although the natural LP relaxation is weak ($r_L$ is large), Gurobi can close much of the gap by presolve and cutting planes at the root node ($r_R$ is small).

\section{Time Spent on Obtaining Different TKP Formulations}\label{apd:time_gen_TKP}
\begin{table}[hbt!]
    \centering
    \scriptsize
    \caption{Comparison of Time Spent on Obtaining Different Formulations for TKP}
    \label{Table:Comparison_GenTime_TKP}
    \begin{tabular}{crrrr}
        \toprule
        Subclass & \multicolumn{4}{c}{Generation Time (s)}\\
        \cmidrule(lr){2-5}
        & {\OBJ}/{\DWB-32} & {\STR-32} & {\OBJ}/{\DWB-64} & {\STR-64}\\
        \midrule\midrule\addlinespace
        XVIII & 78 & 204 & 83 & 772\\
        XIX & 80 & 218 & 82 & 643\\
        XX & 106 & 267 & 101 & 793\\
        \midrule\midrule\addlinespace
    \end{tabular}
\end{table}
In Table \ref{Table:Comparison_GenTime_TKP}, we report time spent on generating formulations {\OBJ}/{\DWB}-$B$ and {\STR}-$B$ for $B\in\{32,64\}$. Unlike Table \ref{Table:Comparison_GenTime} for MKAP, time reported in Table \ref{Table:Comparison_GenTime_TKP} includes time spent on solving the Lagrangian dual problem.

\end{document}